\documentclass[11pt]{amsart}
\usepackage{amssymb,amsmath,amsfonts,amscd,euscript}
\usepackage{hyperref}
\newcommand{\nc}{\newcommand}

\numberwithin{equation}{section}
\newtheorem{thm}{Theorem}[section]
\newtheorem{prop}[thm]{Proposition}
\newtheorem{lem}[thm]{Lemma}
\newtheorem{cor}[thm]{Corollary}
\newtheorem{rem}[thm]{Remark}
\newtheorem{definition}[thm]{Definition}
\newtheorem{example}[thm]{Example}
\newtheorem{dfn}[thm]{Definition}

\nc{\gl}{\mathfrak{gl}}
\nc{\GL}{\mathfrak{GL}}
\nc{\g}{\mathfrak{g}}
\nc{\gh}{\widehat\g}
\nc{\h}{\mathfrak{h}}
\nc{\la}{\lambda}
\nc{\al}{\alpha }
\nc{\be}{\beta }
\nc{\ve}{\varepsilon }
\nc{\om}{\omega }

\nc{\ta}{\theta}
\nc{\ch}{{\mathop {\rm ch}}}
\nc{\Tr}{{\mathop {\rm Tr}\,}}
\nc{\Id}{{\mathop {\rm Id}}}
\nc{\ad}{{\mathop {\rm ad}}}
\nc{\bra}{\langle}
\nc{\ket}{\rangle}
\nc{\x}{{\bf x}}
\nc{\bm}{{\bf m}}
\nc{\bs}{{\bf s}}
\nc{\br}{{\bf r}}
\nc{\bb}{{\bf b}}
\nc{\bk}{{\bf k}}
\nc{\bp}{{\bf p}}
\nc{\pa}{\partial}
\nc{\ld}{\ldots}
\nc{\cd}{\cdots}
\nc{\hk}{\hookrightarrow}
\nc{\T}{\otimes}
\nc{\gr}{\mathrm{gr}}
\nc{\ov}{\overline}
\newcommand{\bin}[2]{\genfrac{(}{)}{0pt}{}{#1}{#2}}

\nc{\cO}{\mathcal O}
\nc{\msl}{\mathfrak{sl}}
\nc{\mgl}{\mathfrak{gl}}
\nc{\U}{\mathrm U}
\nc{\V}{\EuScript V}
\nc{\cL}{\mathcal{L}}
\nc{\Res}{\mathrm{Res\ }}

\newcommand{\bZ}{{\mathbb Z}}

\newcommand{\bP}{{\mathbb P}}

\newcommand{\fh}{{\mathfrak h}}

\newcommand{\fg}{{\mathfrak g}}

\newcommand{\fb}{{\mathfrak b}}

\newcommand{\fn}{{\mathfrak n}}

\begin{document}

\title[Semi-infinite Pl\"ucker relations]
{Semi-infinite Pl\"ucker relations and Weyl modules}

\author{Evgeny Feigin}
\address{Evgeny Feigin:\newline
Department of Mathematics,\newline
National Research University Higher School of Economics,\newline
Usacheva str. 6, 119048, Moscow, Russia,\newline
{\it and }\newline
Skolkovo Institute of Science and Technology, Skolkovo Innovation Center, Building 3,
Moscow 143026, Russia
}
\email{evgfeig@gmail.com}

\author{Ievgen Makedonskyi}
\address{Ievgen Makedonskyi:\newline
Max Planck Institute for Mathematics, Vivatgasse 7, 53111, Bonn, Germany
\newline
{\it and} \newline
Department of Mathematics,\newline
National Research University Higher School of Economics,\newline
Usacheva str. 6, 119048, Moscow, Russia
}
\email{makedonskii\_e@mail.ru}

\begin{abstract}
The goal of this paper is twofold. First, we write down the semi-infinite Pl\"ucker relations, describing
the Drinfeld-Pl\"ucker embedding of the (formal version of) semi-infinite flag varieties in type A. Second,
we study the homogeneous coordinate ring, i.e. the quotient by the ideal generated by the semi-infinite
Pl\"ucker relations. We establish the isomorphism with the algebra of dual global Weyl modules and derive
a new character formula.
\end{abstract}

\maketitle

\section*{Introduction}
The goal of this paper is to study the homogeneous coordinate ring of the formal version of the semi-infinite flag varieties
of type $A$ (see \cite{FiMi}). More precisely, we consider the Drinfeld--Pl\"ucker embedding of the
semi-infinite flag variety, compute explicitly the reduced scheme structure and study the
quotient algebra with respect to the ideal of defining relations. Before giving more detailed summary of our
results, we describe the finite-dimensional analogue of the story (see e.g. \cite{F}).

Let $G=SL_n=SL_n(\mathbb{K})$ ($\mathbb{K}$ is an algebraically closed field of characteristic $0$)
and let $B\subset G$ be a Borel subgroup. The quotient $G/B$ is known to be isomorphic to
the variety of complete flags in an $n$-dimensional vector space. The Pl\"ucker embedding realizes the flag variety
inside the product of projective spaces of all fundamental representations of $SL_n$, which are isomorphic
to the wedge powers of the vector representation in type $A$. In particular, the coordinates $X_I$ on the $k$-th
fundamental representation are labeled by the cardinality $k$ subsets of the set $\{1,\dots,n\}$.
The (quadratic) Pl\"ucker relations describe the image
of this embedding, i.e. they generate the ideal $J_n$ of all multi-homogeneous polynomials vanishing
on the image of the Pl\"ucker embedding. By definition, $J_n$ is the ideal of relations satisfied by
the minors of the matrices from $SL_n$. A very important property of $J_n$ is that the quotient of the polynomial
ring $R_n$ in variables $X_I$ modulo the ideal $J_n$ is isomorphic to the direct sum of (dual) irreducible finite-dimensional
representations of $SL_n$. Finally, let us mention that there exists a remarkable basis of $R_n$ consisting of
monomials in $X_I$. The monomials in this basis are parametrized by the semi-standard tableaux.

Our goal is to generalize the finite-dimensional picture to the semi-infinite settings \cite{FiMi,BF1,BF2,Kat}.
This means that the group $G$ is replaced with the group $G[[t]]$ and all the representations $V$ are replaced with
the infinite-dimensional spaces $V[[t]]=V\otimes \mathbb{K}[[t]]$. The main geometric object ${\bf Q}$ we are interested in
is defined as follows. Let $V(\omega_k)$, $k=1,\dots,n-1$ be the fundamental irreducible modules of $SL_n$.
Let $\mathring{\bf Q}\subset \prod_{k=1}^{n-1} \bP(V(\omega_k)[[t]])$ be the $G[[t]]$-orbit through the product
of highest weight lines (with the reduced scheme structure). Then the scheme ${\bf Q}$ is defined as the Zariski
closure of $\mathring{\bf Q}$ inside the
product of projective spaces. We note that the (formal) Drinfeld-Pl\"ucker data (see \cite{FiMi}) defines the
(non-reduced) scheme with the same support (a.k.a. the projectivization of the arc scheme of the closure of
the basic affine scheme). The scheme ${\bf Q}$ was studied in \cite{KNS}. In particular, it was shown that
the multi-homogeneous coordinate ring of ${\bf Q}$ is isomorphic to the direct sum of dual global Weyl modules \cite{CFK,CI}.
The global Weyl modules are (infinite dimensional) cyclic representations of the current algebra ${\rm Lie}(G)\T \mathbb{K}[t]$ defined
by the condition that the constant part generate finite dimensional ${\rm Lie}(G)$ module with the prescribed highest weight.
Our main results are as follows:
\begin{itemize}
\item we determine the reduced scheme structure of ${\bf Q}$, i.e. we write down the generators of the ideal of relations
of $\mathring{\bf Q}$ (Proposition \ref{propsnakeplueckerequation});
\item we derive a new explicit formula for the characters of the Weyl modules (Theorem \ref{main});
\item we show that the homogeneous coordinate ring of ${\bf Q}$ is isomorphic to the direct sum of dual global Weyl modules
(Corollary \ref{cormain});
\item we give a monomial basis of the homogeneous coordinate ring; the labeling set for the monomials extends the
semi-standard tableaux (Corollary \ref{corssst}).
\end{itemize}
In particular, Corollary \ref{cormain} gives a new proof in type $A$ of the Kato-Naito-Sagaki theorem on the homogeneous
coordinate ring of the semi-infinite flag varieties. Our approach is very different compared to the one used in \cite{KNS}:
we write down explicitly the ideal of relations for the Drinfeld-Pl\"ucker embedding of ${\bf Q}$ and describe the quotient.

Our main players are two algebras. The first algebra ${\mathcal M}$ is generated by the coefficients of minors of the elements
of the group
$G[[t]]$ (to be precise, we consider the minors supported on the first $k$ rows for some $k$; each such a minor is a Taylor
series in $z$ and ${\mathcal M}$ is generated by all the coefficients of these series). The second algebra
${\mathbb W}$ is defined as the direct sum of dual global Weyl modules for all dominant integral weights (see \cite{Kat}).
Our strategy is as follows: we first write down an explicit set of semi-infinite Pl\"ucker relations for the algebra
${\mathcal M}$, then  show that ${\mathcal M}$ surjects onto  ${\mathbb W}$ and finally prove that
the character of ${\mathbb W}$ coincides with the character of the quotient by the ideal generated by the
semi-infinite Pl\"ucker relations.

Let us add several comments on our results. First, we recall that the classical Pl\"ucker relations are linear combinations of
the quadratic monomials of the form $X_{I_1}X_{I_2}$.
We note that the coordinates on the semi-infinite space $V(\omega_k)[[t]]$ are of the form $X_I^{(l)}$, where $I$ is
a cardinality $k$ subset of $\{1,\dots,n\}$ and $l$ is a non-negative integer. To a variable $X_I$ we attach the generating
function $X_I(s)=\sum_{l\ge 0} X_{I}^{(l)}s^l$. Then a part of semi-infinite Pl\"ucker relations can be obtained in the following
way: one takes a classical Pl\"ucker relation, replaces each variable $X_I$ with the generating function $X_I(s)$ and
collect all coefficients of the resulting formal series in the variable $s$ (this is exactly what one is doing in order to pass
to the arc scheme, see e.g. \cite{Mu}). Then these coefficients belong to the
ideal vanishing on ${\mathring{\bf Q}}$. However, we show that these relations do not generate the whole vanishing ideal.
The complete list of generators is given in Theorem \ref{minorplueckerequation}.

Second, recall that in the classical theory an important role is played by the semi-standard Young tableaux.
In short, a tableau is semi-standard if its columns increase from left to right (the longer columns are on the left).
To compare two columns one compares the entries belonging to one row: an entry of the smaller column
can not exceed the corresponding entry of the larger one (for each row). In particular, this order is partial and
there are many uncomparable pairs of columns. We introduce a measure $k(\sigma,\tau)$ of how much uncomparable
the columns $\sigma$ and $\tau$ are. Roughly speaking, the $k(\sigma,\tau)$ is equal to the number of times
the signs $"<"$ and $">"$ got changed when one compares the entries of $\sigma$ and $\tau$ row by row moving from bottom to
top. The complete definition is given in Definition \ref{kdefinition}. The quantity $k(\sigma,\tau)$ is used in our
description of the semi-infinite Pl\"ucker relations as well as in the character formula for the Weyl
modules.

Third, let us comment on the character formula for the global Weyl modules. The characters of global (infinite-dimensional)
and local (finite-dimensional) Weyl modules
differ by a simple factor (\cite{CFK,N}). We derive a character formula for the local Weyl modules for the Lie algebra $\msl_n$
by embedding it into the Weyl module for $\msl_{2n}$. We give a construction of a basis of this embedded module and
compute its character.
The character of the local Weyl module $W(\mu)$ coincides with the specialization of nonsymmetric Macdonald polynomial
$E_{w_0(\mu)}(X,q,0)$ (see \cite{M,Ch,N,S,I,CI}). Thus we obtain a combinatorial formula for such specializations for arbitrary
dominant  weight $\mu$. Note that in
\cite{FM} we specialized the combinatorial formula of Haglund, Haiman ond Loehr (see \cite{HHL}) and
gave a recurrent formula for the nonsymmetric Macdonald polynomials at $t=0$. However, we were
able to obtain an explicit character formula only for the Weyl modules $W(\mu)$, where $\mu$ is a linear combination
of $\omega_1$ and $\omega_{n-1}$.

Finally let us add a remark due to Michael Finkelberg.
In the works \cite{BF1,BF2,BF3}, the authors study the rings of functions on zastava schemes (for the curve ${\mathbb A}^1$) and the spaces of
sections of line bundles over quasimaps' schemes (for the curve ${\mathbb P}^1$); in particular, the characters of these spaces are computed. The results
of this paper imply that these zastava and quasimaps as schemes representing the corresponding moduli problems, are {\em non reduced} (for $G=SL(n),\ n\geq5$).
The zastava and quasimaps' spaces studied in \cite{BF1,BF2,BF3} are defined as the corresponding {\em reduced} schemes, i.e.\ {\em varieties}. In particular,
the character formulas of \cite{BF1,BF2,BF3} hold for the spaces of sections over the {\em reduced} zastava and quasimaps' varieties.

The paper is organized as follows. In Section \ref{generalities} we collect the notation and give generalities on the
representation theoretic,
combinatorial and geometric structures we use in the paper. In Section \ref{Pluecker} we write down the generators of the
ideal of relations for the semi-infinite flag varieties; as a consequence we derive an estimate for the character of the
global Weyl modules. In Section \ref{Weyl} we prove that formula from Section \ref{Pluecker} for the characters of the
Weyl modules holds and thus finalize the proof. In Appendix \ref{app} we describe the whole story for the simplest symplectic algebra
$\mathfrak{sp}_4$. At the moment we are not able to work out the case of general $\mathfrak{sp}_{2n}$, since the
Pl\"ucker relations in type $C$ are more complicated than in type $A$.

\section{Generalities}\label{generalities}
\subsection{Finite-dimensional picture}
Let $\mathfrak{g}$ be a simple finite dimensional Lie algebra over an algebraically closed field $\mathbb{K}$ of characteristic $0$.
Let $\Delta$ be the root system of $\mathfrak{g}$, $\Delta=\Delta_+\sqcup\Delta_-$ the union of positive and negative roots.
Let $\mathfrak{g}=\mathfrak{n}_-\oplus \mathfrak{h} \oplus \mathfrak{n}_+$ be the Cartan decomposition of $\fg$.
For a positive root $\al$ let $e_\al\in\fn_+$ and $f_{-\al}\in\fn_-$ be the Chevalley generators.
The weight lattice $X$ contains the positive part $X_+$, containing all fundamental weights $\omega_k$, $k=1,\dots,{\rm rk}(\fg)$.
For $\la\in X_+$ we denote by $V(\la)$ the irreducible
highest weight $\fg$-module with highest weight $\la$.

For any two weights $\lambda, \mu\in X_+$ there is an embedding
\[V(\lambda+\mu) \hookrightarrow V(\lambda)\otimes V(\mu).\]
Dualizing this injection we obtain the surjection:
\[V(\lambda)^*\otimes V(\mu)^* \twoheadrightarrow V(\lambda+\mu)^*.\]
These surjections define a multiplication on the space $\mathbb{V}=\bigoplus_{\lambda \in X_+}V(\lambda)^*$
and make $\mathbb{V}$ into an associative and commutative algebra.

Take the corresponding algebraic groups $G\supset B \supset U$, where $B$ is a Borel subgroup and
$N$ is a unipotent subgroup. Then we have:
\[G/U\simeq {\rm Spec}\mathbb{V}.\]
The algebra $\mathbb{V}$ has the natural $X_+$-grading. The projective spectrum of $\mathbb{V}$ with respect to
this grading is isomorphic to the flag variety $G/B$ (see e.g. \cite{Kum}).

Let $G=SL_n$, $\fg=sl_n$. We take $n^2$ variables $z_{ij}$, $1 \leq i,j \leq n$ and consider
the subalgebra of $\mathbb{K}[z_{i,j}]_{i,j=1}^n$ generated by determinants
\[\left|\left(
                                                       \begin{array}{cccc}
                                                         z_{1i_1} & z_{1i_2} & \cdots & z_{1i_k} \\
                                                         z_{2i_1} & z_{2i_2} & \cdots & z_{2i_k} \\
                                                         \vdots & \vdots & \ddots & \cdots \\
                                                         z_{ki_1} & z_{ki_2} & \cdots & z_{ki_k} \\
                                                       \end{array}
                                                     \right)
\right|\]
for all $k\ge 1$, $1 \leq i_1<i_2<\dots <i_k \leq n$. Then this algebra is isomorphic to $\mathbb{V}$.
The algebra  $\mathbb{V}$ is generated by dual fundamental modules $V(\omega_k)^*$ which satisfy (quadratic) Pl\"ucker relations.
There is a remarkable basis of a module $V(\lambda)^*$ given by semistandard tableaux of shape $\lambda$
(see \cite{F}).

In Section \ref{Pluecker} we describe the semi-infinite analogue of these constructions.

\subsection{Weyl modules}
In this paper we are mainly interested in representations of the current algebra $\fg \otimes \mathbb{K}[t]$
which is a maximal parabolic subalgebra of the affine Kac-Moody Lie algebra attached to $\fg$.
For $x\in\fg$ we sometimes denote the element $x\T 1$ simply by $x$.

\begin{definition}\label{weylmodules}\cite{CP}
Let $\lambda\in X_+$. Then the global Weyl module
$\mathbb{W}(\lambda)$ is the cyclic $\fg \otimes \mathbb{K}[t]$ module with a generator $v_\lambda$ and the following
defining relations:
\begin{gather}
\label{weylvanishing1}
(e_{\alpha}\otimes t^k) v_{\lambda}=0, \ \alpha \in \Delta_+, ~k \geq 0;\\
\label{weylbound1}
(f_{-\alpha}\otimes 1)^{\langle \alpha^\vee, \lambda \rangle+1} v_{\lambda}=0, \ \alpha \in \Delta_+.
\end{gather}
Local Weyl modules ${W}(\lambda)$ are defined by previous conditions and one additional condition:
\begin{equation}
\label{weylvanishing0}
h\T t^k v_{\lambda}=0 \text{ for all } h\in\fh, k>0.\\
\end{equation}
\end{definition}

Weyl modules are the natural analogues of finite-dimensional simple $\fg$-modules $V(\lambda)$. They are graded by the degree of $t$:
\[\mathbb{W}(\mu)=\bigoplus_{k=0}^\infty \mathbb{W}(\mu)^{(k)}\]
with finite-dimensional homogeneous components. Therefore we can define the restricted dual module:
\[\mathbb{W}(\mu)^*=\bigoplus_{k=0}^\infty (\mathbb{W}(\mu)^{(k)})^*.\]

Global Weyl modules have the following properties.

\begin{lem}\label{cocyclic}
$\mathbb{W}(\mu)^*$ is cocyclic, i. e. there exists an element (cogenerator) $v^*\in\mathbb{W}(\mu)^*$
such that for any element $u\in\mathbb{W}(\mu)^*$
there exists an element $f \in U(\fg \otimes \mathbb{K}[t])$ such that $fu=v^*$. The set of cogenerators coincides
with $({\mathbb{W}(\mu)^{(0)}})^*\simeq V(\mu)^*$ .
\end{lem}
\begin{proof}
This is a direct consequence of the fact that Weyl module is cyclic.
\end{proof}

For a dominant weight $\lambda=\sum_{k=1}^{{\rm rk}(\fg)} m_k\omega_k$ we define
\[(q)_\lambda=\prod_{k=1}^{{\rm rk}}\prod_{i=1}^{m_k}(1-q^i).\]

Each Weyl module is graded by the $\fh$-weights and by $t$-degree. For any $\fg\otimes \mathbb{K}[t]$-module $U$ with such a grading
let $U(\nu,m)$, $\nu\in\fh^*$, $m\in\bZ$ be the weight space of the corresponding weight.
\begin{definition}
\[ \ch U=\sum_{\nu,m}\dim U(\nu,m)x^{\nu}q^m.\]
\end{definition}

\begin{prop}\cite{CFK,N}\label{localglobal}
\[\ch \mathbb{W}(\mu)=\frac{\ch {W}(\mu)}{(q)_\mu}.\]
\end{prop}

\begin{lem}\label{weylinjection}\cite{Kat}
The $\fg \otimes \mathbb{K}[t]$-submodule of $\mathbb{W}(\la)\otimes\mathbb{W}(\mu)$
generated by $v_{\lambda}\otimes v_{\mu}$ is
isomorphic to $\mathbb{W}( \lambda+\mu)$.
\end{lem}

\begin{cor}\label{weylmodulering}
There exists a surjection of dual Weyl modules:
\begin{equation}\label{dualweylproduct}
\mathbb{W}^*(\la)\otimes\mathbb{W}^*(\mu)\twoheadrightarrow\mathbb{W}^*(\la+ \mu),
\end{equation}
inducing the structure of associative and commutative algebra on the space $\bigoplus_{\lambda \in X_+}\mathbb{W}^*(\lambda)$.
We denote this algebra by $\mathbb{W}=\mathbb{W}(\fg)$.
\end{cor}

\begin{rem}
The algebra $\mathbb{W}$ is an analogue of the algebra $\mathbb{V}=\bigoplus_{\lambda \in X_+}{V}(\lambda)^*$.
The algebra $\mathbb{V}$ is generated by the space $\bigoplus_{k=1}^{\rm rk (\fg)}V({\omega_k})^*$ and there are only
quadratic relations (see \cite{F,Kum}).
\end{rem}

The following proposition is a direct consequence of Corollary \ref{weylmodulering}.

\begin{prop}
$\mathbb{W}$ is generated by the space $\bigoplus_{k=1}^{{\rm rk}(\mathfrak{g})}\mathbb{W}(\omega_k)^*$.
\end{prop}

Let us describe the structure of fundamental Weyl modules $\mathbb{W}(\omega_k)$.

\begin{lem}\label{fundamentalstructure}
Assume that $W(\omega_k)|_{\fg \otimes 1}\simeq V(\omega_k)$. Then for the global Weyl module we have:
\[\mathbb{W}(\omega_k)\simeq V(\omega_k) \otimes \mathbb{K}[t]\]
with the action of the current algebra given by the following rule:
\[x \otimes t^l.u \otimes t^k=x.u \otimes t^{l+k}.\]
\end{lem}
\begin{proof}
We note that all defining relations for $\mathbb{W}(\omega_k)$ hold on the constructed module. Therefore we have a surjection
$\mathbb{W}(\omega_k)\twoheadrightarrow V(\omega_k) \otimes \mathbb{K}[t].$
Now it suffices to note that
\[\ch \left( V(\omega_k) \otimes \mathbb{K}[t]\right) = \frac{\ch W(\omega_k)}{(q)_{\omega_k}}= \frac{\ch W(\omega_k)}{1-q}.\]
\end{proof}

\begin{lem}\label{dualfundamentalstructure}
Under assumptions of Lemma \ref{fundamentalstructure} we have:
\[\mathbb{W}(\omega_k)^*\simeq V(\omega_k)^* \otimes \mathbb{K}[t]\]
with the action of the current algebra given by the following rule:
\begin{equation}
x \otimes t^l.u \otimes t^k=
                            \begin{cases}
                             x.u \otimes t^{k-l}, & \text{if}~ k \geq l;\\
                             0, & \text{otherwise.} \\
                            \end{cases}
\end{equation}
\end{lem}

\begin{rem}\label{goodfundamental}
Note that the conditions of Lemma \ref{fundamentalstructure} hold for all fundamental weights in types $A$ and $C$,
vector and spinor
representations in types $B$ and $D$, two $27$-dimensional representations in type $E_6$, $56$-dimensional representation in type $E_7$, $26$-dimensional representation in type $F_4$ and $7$-dimensional representation in type $G_2$.
\end{rem}

From now on we assume that $W(\omega_k)|_{\fg \otimes 1}\simeq V(\omega_k)$ for all fundamental weights.
Let $\{X_{i,1},\dots, X_{k,l_k}\}$ be a basis of $V(\omega_k)^*$. Then $\{X_{kj}^{(l)}=X_{kj}\otimes t^l\}$ is a basis of
$\mathbb{W}(\omega_k)^*$.
We consider the formal series $X_{kj}(s)=\sum_{l=0}^\infty X_{kj}^{(l)} s^l$.

\begin{prop}\label{classicaltosemiinfrelations}
Assume that a relation
$$r:=\sum_{\substack{j_1=1,\dots,l_{k_1}\\j_2=1,\dots,l_{k_1}}} c_{j_1,j_2} X_{k_1,j_1}X_{k_2,j_2}=0$$
holds in $\mathbb{V}$ for some constants $c_{j_1,j_2}\in\mathbb{K}$. Then

$$r(s):=\sum_{\substack{j_1=1,\dots,l_{k_1}\\j_2=1,\dots,l_{k_1}}} c_{j_1,j_2} X_{k_1,j_1}(s)X_{k_2,j_2}(s)=0$$
in the algebra $\mathbb{W}[[s]]$,
i. e. all the coefficients of this series form relations in $\mathbb{W}$.
\end{prop}
\begin{proof}
We know that for any relation $r$ and any element $e \in \fg$ the element $e.r=r_1$ is a relation. We have $e \otimes 1.r(s)=r_1(s)$.
Let $r(s)=\sum_{l=0}^\infty r^{(l)} s^l$. If
$h_\alpha.r=ar$, $a \in \mathbb{K}$, then $h_{\alpha}\otimes t^l.r^{(k)}=ar^{(k-l)}$. Thus all coefficients of series of the form $r(s)$
form a submodule in $\mathbb{W}(\omega_{k_1}+\omega_{k_2})^*$. Note that this submodule has zero intersection with the zero level submodule
${V(\omega_{k_1}+\omega_{k_2})}^*$. However $\mathbb{W}^*(\omega_{k_1}+\omega_{k_2})$ is cocyclic. Therefore this submodule is zero.
\end{proof}

However in general these are not only relations in the algebra $\mathbb{W}$.

\subsection{Semi-infinite flag varieties.}
Let $G[[t]]$ be the group of the Lie algebra $\fg \otimes \mathbb{K}[[t]]$, $B[[t]]\supset U[[t]]$ be the groups of
$\fb \otimes \mathbb{K}[[t]]\supset \fn \otimes \mathbb{K}[[t]]$ respectively. In this paper we deal with
$\fg=\msl_n$. In this case $G[[t]]=SL(\mathbb{K}[[t]])$,
$B[[t]]$ and $U[[t]]$ are upper triangular and unitriangular matrices over this ring. We consider the varieties
$G[[t]]/ U[[t]]$ and $G[[t]]/H \cdot U[[t]]$. They have the following realization due to Drinfeld \cite{FiMi,BG}.

For any $\lambda \in X_+$ consider the space $V({\lambda})[[t]]=V({\lambda})\otimes \mathbb{K}[[t]]$. Then we have the family of embeddings
$m_{\lambda,\mu}: V({\lambda+\mu})[[t]]\hookrightarrow V({\lambda})[[t]]\otimes V({\mu})[[t]] $. A Drinfeld-Pl\"ucker data is a set of
lines $l_{\lambda} \in V(\lambda)[[t]]$ such that $m_{\lambda,\mu}l_{\lambda+\mu}=l_{\lambda}\otimes l_{\mu}$. Such a collection
of lines is fixed
by the lines  $l_{\omega_k} \in V(\omega_k)[[t]]$ for all fundamental weights $\omega_k$. Take a basis $\{v_{kj}\}$ of
$V(\omega_k)$.
Put $l_{\omega_k}=\mathbb{K}\sum_j a_{kj}v_{kj}$, where $a_{kj}\in \mathbb{K}[[t]]$.
Then the family of lines $\{l_{\omega_k}\}$ gives a Drinfeld-Pl\"ucker
data iff the series $(a_{kj})$ satisfy the Pl\"ucker relations (thus this realization of $\overline{G[[t]]/ U[[t]]}$
is the projectivized arc scheme of $\overline{G/U}$, see e.g. \cite{Mu}).
The set of points satisfying these relations contains $G[[t]]/H \cdot U[[t]]$
as an open dense subset. Indeed, take a product of the highest weight lines
$\mathbb{K} v_{\omega_1}\otimes \dots \otimes \mathbb{K} v_{\omega_{n-1}}$ in
$V(\omega_1)[[t]]\otimes \dots \otimes V(\omega_{n-1})[[t]]$. Then the set of Drinfeld-Pl\"ucker data is the closure of the
$G[[t]]$-orbit of this element and its stabilizer is $H \cdot U[[t]]$.

In type $A$ this construction can be written down in a very explicit way. We denote the coordinates in $V(\omega_k)[[t]]$ by
$X_I^{(l)}$, where $I=(1\le i_1<\dots<i_k<n)$ and $l\in\bZ_{\ge 0}$
(recall that $V(\omega_k)$ is the $k$-th wedge power of the vector representation; the coordinate
$X_I^{(l)}$ is dual to $v_{i_1}\wedge\dots\wedge v_{i_k}\otimes t^l$).
An element $g\in SL_n[[t]]$ is given
by $n\times n$ matrix whose entries are Taylor series in $t$. Therefore a coordinate $X_I^{(l)}$ of a point
$g\cdot \mathbb{K} v_{\omega_1}\otimes \dots \otimes \mathbb{K} v_{\omega_{n-1}}$ is equal to
the coefficient in front of $t^l$ of the minor of the matrix $g$ located on the intersection of the first $k$
rows and $i_1,\dots,i_k$ columns.

Finally, let us note that the Pl\"ucker relations $m_{\lambda,\mu}l_{\lambda+\mu} =l_{\lambda}\otimes l_{\mu}$
give the non-reduced scheme structure on the $G[[t]]$-orbit closure for $n\ge 5$.
Corollary 4.27 from \cite{KNS}
says that $\overline{G[[t]]/ U[[t]]}\simeq {\rm Spec}\mathbb{W}$ and $\overline{G[[t]]/H \cdot U[[t]]}$ is the spectrum of multi-homogenous
ideals of ${\mathbb W}$. In Section \ref{Pluecker} we describe explicitly the difference between the reduced scheme structure
provided by ${\mathbb W}$ and the non-reduced scheme structure defined by $\mathbb{K}[t]$ Pl\"ucker relations. We close this section with
an example.

\begin{example}
Let $G=SL_5$. Then one can check that the coordinates of a point in the open $G[[t]]$-orbit satisfy the relation
\begin{multline}\label{example}
X_{12}^{(1)}X_{345}^{(0)}-X_{13}^{(1)}X_{245}^{(0)}+X_{14}^{(1)}X_{235}^{(0)} - X_{15}^{(1)}X_{234}^{(0)}+ X_{23}^{(1)}X_{145}^{(0)}\\
- X_{24}^{(1)}X_{135}^{(0)} + X_{25}^{(1)}X_{134}^{(0)} + X_{34}^{(1)}X_{125}^{(0)} - X_{35}^{(1)}X_{124}^{(0)} + X_{45}^{(1)}X_{123}^{(0)}.
\end{multline}
We note that \eqref{example} does not belong to the ideal generated by the coefficients of $\mathbb{K}[t]$ Pl\"ucker relations, because
(in particular) in each Pl\"ucker relation the term
$X_{12}^{(1)}X_{345}^{(0)}$ shows up together with its companion $X_{12}^{(0)}X_{345}^{(1)}$. For example,
the sum of \eqref{example} with
\begin{multline*}
X_{12}^{(0)}X_{345}^{(1)}-X_{13}^{(0)}X_{245}^{(1)}+X_{14}^{(0)}X_{235}^{(1)}- X_{15}^{(0)}X_{234}^{(1)}+ X_{23}^{(0)}X_{145}^{(1)}\\
- X_{24}^{(0)}X_{135}^{(1)} + X_{25}^{(0)}X_{134}^{(1)} + X_{34}^{(0)}X_{125}^{(1)} - X_{35}^{(0)}X_{124}^{(1)} + X_{45}^{(0)}X_{123}^{(1)}.
\end{multline*}
does belong to the ideal of Pl\"ucker relations.
\end{example}

\section{Semi-infinite Pl\"ucker relations}\label{Pluecker}
\subsection{Algebra of minors.}

Consider the set of variables $z_{ij}^{(l)}$, $1 \leq i,j \leq n$, $l \geq 0$.
Let $e_{pq}$ be the $(p,q)$-th matrix unit, $h_{pq}=e_{pp}-e_{qq}$.
We define a derivation action of $\msl_n \otimes \mathbb{K}[t]$ on the polynomials in these variables by the following action of generators:
\[e_{qp}\otimes t^l z_{ij}^{(k)}=
                            \begin{cases}
                            -z_{ip}^{(k-l)},& \text{if}~ q=j,~ k \geq l;\\
                             0, &\text{otherwise.} \\
                            \end{cases},\ p\ne q\]
\[h_{qp}\otimes t^l z_{ij}^{(k)}=
                            \begin{cases}
                            -z_{ij}^{(k-l)}, & \text{if }  p=j,~ k \geq l;\\
                            z_{ij}^{(k-l)},  & \text{if }  q=j,~ k \geq l;\\
                             0, & \text{otherwise.}
                            \end{cases}\]

We define the formal series $z_{ij}(s)=\sum_{k=0}^\infty z_{ij}^{(k)} s^k$.
Let
\begin{equation}
m_{i_1, \dots, i_k}(s)=\det \left(
                                 \begin{array}{cccc}
                                   z_{1,i_1}(s) & z_{1,i_2}(s) & \cdots & z_{1,i_k}(s) \\
                                   z_{2,i_1}(s) & z_{2,i_2}(s) & \cdots & z_{2,i_k}(s) \\
                                   \vdots & \vdots & \ddots & \vdots \\
                                   z_{k,i_1}(s) & z_{k,i_2}(s) & \cdots & z_{k,i_k}(s) \\
                                 \end{array}
                               \right)
\end{equation}
Note that $m_{i_1, \dots, i_k}(s)$ is alternating with respect to permutation of indexes.

Let us write the decomposition $m_{i_1, \dots, i_k}(s)=\sum_{l=0}^{\infty}m_{i_1, \dots, i_k}^{(l)} s^l$.
We consider the subalgebra
$\mathcal{M}=\mathbb{K}[m_{i_1, \dots, i_k}^{(l)}]\subset \mathbb{K}[z_{ij}^{(l)}]$ and define $\deg m_{i_1, \dots, i_k}^{(l)}=\omega_k$.
Then $\mathcal{M}$ is graded by the semigroup of dominant weights:
\[\mathcal{M}=\bigoplus_{\lambda \in X_+}\mathcal{M}(\lambda).\]
Each $\mathcal{M}(\lambda)$ is $\msl_n\otimes \mathbb{K}[t]$-submodule by the degree of $t$,
\[\mathcal{M}(\lambda)=\bigoplus_{l=0}^\infty\mathcal{M}(\lambda)^{(l)}.\]

\begin{prop}\label{fundamentalasminors}
The set of elements
\[\{m_{i_1, \dots, i_k}^{(l)}, ~i_1<i_2<\dots<i_k, ~l=0, 1, \dots\}\]
form a basis of $\msl_n \otimes \mathbb{K}[t]$-module $\mathcal{M}(\omega_k)\simeq \mathbb{W}(\omega_{k})^*$.
\end{prop}
\begin{proof}
 We have the standard basis of $ V(\omega_k)^*\simeq V(\omega_{n-k})$ consisting of elements $\{X_{i_1, \dots, i_k}\}$ defined by:
\[X_{i_1, \dots, i_k}=v_{i_1}^*\wedge \dots \wedge v_{i_k}^*, i_1 < \dots <i_k.\]
Therefore  Lemma \ref{dualfundamentalstructure} tells us that $\{X_{i_1, \dots, i_k}\otimes t^{l}\}, l=0,1,\dots$ is a basis of
$\mathbb{W}(\omega_{k})^*$.
It is easy to see that the map $\{X_{i_1, \dots, i_k}\otimes t^{l}\}\mapsto \{m_{i_1, \dots, i_k}^{(l)}\}$
defines the surjection of modules. Moreover the set $\{m_{i_1, \dots, i_k}^{(0)}\}$
forms a basis of $V(\omega_k)^*$ and the dual Weyl module is cocyclic. Thus all the elements
$m_{i_1, \dots, i_k}^{(l)}$ are linearly independent.
\end{proof}

Let $X_{i_1, \dots, i_p, i_{p+1}, \dots i_k}=-X_{i_1, \dots, i_{p+1}, i_{p}, \dots i_k}$, $X_{i_1, \dots, i_p, i_{p}, \dots i_k}=0$.
Using this rule we define $X_{I}$ for arbitrary $k$-tuple $I$.

\begin{prop}\label{frommatrixtoweyl}
There exists the surjection of algebras $\mathcal{M}\twoheadrightarrow \mathbb{W}$ sending $m^{(l)}_{i_1,\dots,i_k}$ to
$X^{(l)}_{i_1,\dots,i_k}$ and $\mathcal{M}(\lambda)$ to ${\mathbb{W}(\lambda)}^*$.
\end{prop}
\begin{proof}
Degree zero elements $m_{i_1, \dots, i_k}^{(0)}$ satisfy usual Pl\"ucker relations (see, for example \cite{MS}).
Therefore degree zero submodule $\mathcal{M}(\lambda)^{(0)}$ is isomorphic to $V(\lambda)^*$.
Consider a polynomial algebra generated by auxiliary variables $\chi_{i_1, \dots, i_k}^{(l)}$. We define the action of
$\msl_n \otimes \mathbb{K}[t]$
on $\langle \chi_{i_1, \dots, i_k}^{(l)}\rangle$ in such a way that the map $\chi_{i_1, \dots, i_k}^{(l)}\mapsto m_{i_1, \dots, i_k}^{(l)}$
is an isomorphism of $\langle \chi_{i_1, \dots, i_k}^{(l)}\rangle$ and $\bigoplus_{k=1}^{{\rm rk} (\fg)}\mathcal{M}(\omega_k)$.
We attach the homogeneous degree $\omega_k$ and t-degree $l$ to any variable of the form $\chi_{i_1, \dots, i_k}^{(l)}$.
Then we have the following surjections of algebras with derivations:
\[\phi:\mathbb{K}[\chi_{i_1, \dots, i_k}^{(l)}] \twoheadrightarrow \mathcal{M};\]
\[\psi:\mathbb{K}[\chi_{i_1, \dots, i_k}^{(l)}] \twoheadrightarrow \mathbb{W}.\]
Let $F_\lambda$ be the submodule of $\mathbb{K}[\chi_{i_1, \dots, i_k}^{(l)}]$ of degree $\lambda$.
Take a homogeneous polynomial $P\in \mathbb{K}[\chi_{i_1, \dots, i_k}^{(l)}]$ of degree $\lambda$. Assume that $\phi(P)=0$.
Then $U(\msl_n \otimes \mathbb{K}[t]).P\cap F_\lambda^0$ does not contain an element of weight $\lambda$. Therefore
$\psi(P)=0$ because $\mathbb{W}(\lambda)^*$ is cocyclic with cogenerator of weight $\lambda$ in t-degree $0$.
\end{proof}

Now we construct a family of (quadratic) relations for the elements $m_{i_1, \dots, i_k}^{(l)}$.
We take two sets of numbers $\sigma=\{i_1, \dots, i_q\}$ and
$\tau=\{j_1, \dots, j_p\}$, $i_1<i_2 \dots,<i_q$,
$j_1<j_2 \dots< j_p$, $q\leq p$. For $k\in \mathbb{N}$ we consider $p+k=a+b$ pairwise distinct numbers
$i_{f_1}, \dots, i_{f_a}$, $j_{g_1}, \dots, j_{g_b}$. Let
\[
P=\{i_{f_1}, \dots, i_{f_a}\}\sqcup\{j_{g_1}, \dots, j_{g_b}\}.
\]
Denote $\sigma\backslash P=\{i_{f'_1}, \dots, i_{f'_{q-a}}\}$,
$\tau\backslash P=\{j_{g'_1}, \dots, j_{g'_{p-b}}\}$.
For any cardinality $a$ subset $A \subset P$,
\[
A=\{y_1, \dots, y_a\},\ P\backslash A=\{y_{a+1}, \dots, y_{a+b}\}
\]
with $y_1<\dots <y_a$ and $y_{a+1}<\dots <y_{a+b}$ we consider
\[\eta_1(\sigma,\tau,P,A)=(i_{f'_1}, \dots, i_{f'_{q-a}},y_1, \dots, y_a),\]
\[\eta_2(\sigma,\tau,P,A)=(j_{g'_1}, \dots, j_{g'_{p-b}},y_{a+1}, \dots, y_{a+b}).\]
Denote by ${\rm sign}(\sigma,\tau,P,A)$ the sign of the permutation shuffling the tuple
$(y_1, \dots, y_a,y_{a+1}, \dots, y_{a+b})$ into $\{i_{f_1}, \dots, i_{f_a},j_{g_1}, \dots, j_{g_b}\}$.
\begin{prop}\label{minorplueckerequation}
For any $0 \leq k'\leq k-1$ we have the following equality in $\mathbb{W}[[s]]$:
\begin{equation}\label{plueckerequation}
\sum_{A\subset P, |A|=a}(-1)^{{\rm sign}(\sigma,\tau,P,A)}
\frac{\partial^{k'} m_{\eta_1(\sigma,\tau,P,A)}(s)}{\partial s^{k'}} m_{\eta_2(\sigma,\tau,P,A)}(s)=0.
\end{equation}
\end{prop}
\begin{proof}
For any minor $m_{i_1, \dots, i_q}(s)$ we have the Leibniz rule:
\begin{equation}
\frac{\partial^{k'} m_{i_1, \dots, i_q}(s)}{\partial s^{k'}}=
\sum_{k_1+\dots+k_q=k'}\det \left(
                                 \begin{array}{cccc}
                                   \frac{\partial^{k_1} z_{1,i_1}(s)}{\partial s^{k_1}} & \frac{\partial^{k_1}z_{1,i_2}(s)}{\partial s^{k_1}} & \cdots & \frac{\partial^{k_1}z_{1,i_q}(s)}{\partial s^{k_1}} \\
                                   \frac{\partial^{k_2} z_{2,i_1}(s)}{\partial s^{k_2}} & \frac{\partial^{k_2}z_{2,i_2}(s)}{\partial s^{k_2}} & \cdots & \frac{\partial^{k_2}z_{2,i_q}(s)}{\partial s^{k_2}} \\
                                   \vdots & \vdots & \ddots & \vdots \\
                                   \frac{\partial^{k_q} z_{q,i_1}(s)}{\partial s^{k_q}} & \frac{\partial^{k_q}z_{q,i_2}(s)}{\partial s^{k_q}} & \cdots & \frac{\partial^{k_q}z_{q,i_q}(s)}{\partial s^{k_q}} \\
                                 \end{array}
                               \right).
\end{equation}
Hence equation \eqref{plueckerequation} can be rewritten in the following way:
\begin{multline}\label{changesummorderpluecker}
\sum_{k_1+\dots+k_q=k'} \sum_{A\subset P, |A|=a}(-1)^{{\rm sign}(\sigma,\tau,P,A)}\\ \det \left(
           \begin{array}{cccc}
                \frac{\partial^{k_1} z_{1,i_{f'_1}}(s)}{\partial s^{k_1}} & \cdots & \frac{\partial^{k_1}z_{1,y_a}(s)}{\partial s^{k_1}} \\
                \vdots &  \ddots & \vdots \\
                \frac{\partial^{k_q} z_{q,i_{f'_1}}(s)}{\partial s^{k_q}} & \cdots & \frac{\partial^{k_q}z_{q,y_a}(s)}{\partial s^{k_q}} \\
                \end{array}\right)
								m_{\eta_2(\sigma,\tau,P,A)}(s)=0.
\end{multline}
We consider a summand of the left hand side of equation \eqref{changesummorderpluecker}:
\begin{multline}
S(k_1,\dots,k_q)=\sum_{A\subset P, |A|=a}(-1)^{{\rm sign}(\sigma,\tau,P,A)}\\ \det \left(
                                 \begin{array}{cccc}
                                   \frac{\partial^{k_1} z_{1,i_{f'_1}}(s)}{\partial s^{k_1}} & \cdots & \frac{\partial^{k_1}z_{1,y_a}(s)}{\partial s^{k_1}} \\
                                   \vdots &  \ddots & \vdots \\
                                   \frac{\partial^{k_q} z_{q,i_{f'_1}}(s)}{\partial s^{k_q}} & \cdots & \frac{\partial^{k_q}z_{q,y_a}(s)}{\partial s^{k_q}} \\
                                 \end{array}
                               \right)m_{\eta_2(\sigma,\tau,P,A)}(s).
\end{multline}
We note that there is less than or equal to $k-1$ numbers $j$ such that $k_j>0$. For any $i \in P$ define the column
\[
\left(z_{1,i}(s),\dots, z_{p,i}(s),\frac{\partial^{k_1}z_{1,i}(s)}{\partial s^{k_1}}, \dots, \frac{\partial^{k_q}z_{q,i}(s)}{\partial s^{k_q}}\right)^t,
\]
where the element $\frac{\partial^{k_j}z_{j,i}(s)}{\partial s^{k_j}}$ is skipped if $k_j=0$. All these columns lie in a free
$\mathbb{K}[z_{ij}^{(l)}][[s]]$-module of rank  less than or equal to $p+k-1$.
Consider the space dual to the linear space of these columns.
It is easy to see that $S(k_1,\dots,k_q)$ is multilinear
and alternating function on this space. Indeed take $i, i'\in P$. If for some $A \subset P$: $i,i' \in A$ or $i ,i' \notin A$,
then the corresponding summand is alternating. If $i \in A$, $i' \notin  A$,
then the summand corresponding to the subset $(A\backslash \{i\})\cup \{i'\}$ has another sign.
Therefore this function is equal to $0$ and thus $S(k_1,\dots,k_q)=0$.
\end{proof}

\begin{cor}\label{Weylpluecker}
The following relation holds in $\mathbb{W}$:
\begin{equation}
\sum_{A\subset P, |A|=a}(-1)^{{\rm sign}(\sigma,\tau,P,A)}
\frac{\partial^{k'} X_{\eta_1(\sigma,\tau,P,A)}(s)}{\partial s^{k'}} X_{\eta_2(\sigma,\tau,P,A)}(s)=0.
\end{equation}
\end{cor}

Now we have the set of relations on generators of $\mathbb{W}$. We want to write down a linearly independent set of these relations.
We prepare the following definitions.
\begin{dfn}
For two strictly increasing columns of numbers
\[
\sigma=(\sigma_1, \dots, \sigma_{l_{\sigma}})^t,\quad \tau=(\tau_1, \dots, \tau_{l_{\tau}})^t
\]
we write $\sigma<\tau$ if $l_{\sigma}>l_{\tau}$
 or $l_{\sigma}=l_{\tau}$ and for some $j$ for any $j'>j$
$\sigma_{j'}=\tau_{j'}$ and $\sigma_j<\tau_j$.
\end{dfn}

\begin{dfn}\label{kdefinition}
Assume that for some $\sigma<\tau$ we have the following set of inequalities:
\begin{gather*}
\sigma_{l_{\tau}}\leq\tau_{l_{\tau}}, \dots, \sigma_{j_1+1}\leq\tau_{j_1+1},\\
\sigma_{j_1}>\tau_{j_1}, \sigma_{j_1-1}\geq\tau_{j_1-1}, \dots, \sigma_{j_2+1}\geq\tau_{j_2+1},\\
\sigma_{j_2}<\tau_{j_2},\sigma_{j_2-1}\leq\tau_{j_2-1},\dots.
\end{gather*}
We define strictly decreasing sequence of elements
\[
P(\sigma,\tau)=(\sigma_{l_{\sigma}}, \sigma_{l_{\sigma}-1}, \dots, \sigma_{l_{\tau}+1},
\dots, \sigma_{j_1+1}, \sigma_{j_1}, \tau_{j_1}, \tau_{j_1-1}, \dots, \tau_{j_2}, \sigma_{j_2}, \dots).
\]
We set
\[
k(\sigma,\tau)=|P(\sigma,\tau)|-l_{\sigma}.
\]
\end{dfn}


\begin{example}
In the following examples the numbers $k(\sigma,\tau)$ are equal to $2$, $1$ and $0$ respectively. The elements of the sets $P(\sigma,\tau)$ are
highlighted by the boldface font and are given by $(7,6,5,4,3,1)$, $(7,5,4,3,1)$, $(7,5,3,2)$ respectively.
\label{snakeexample}
\[
\begin{tabular}{ccc}
${\bf 1}$ & $\le $ & $2$\\
${\bf 3}$ &  $<$  & ${\bf 4}$ \\
${\bf 6}$ & $>$ & ${\bf 5}$\\
${\bf 7}$ & $<$ & ${8}$\\
\end{tabular},\qquad
\begin{tabular}{ccc}
${2}$ & $\ge $ & ${\bf 1}$\\
${\bf 4}$ &  $>$  & ${\bf 3}$ \\
${\bf 5}$ & $\le $ & ${6}$\\
${\bf 7}$ &  &
\end{tabular},\qquad
\begin{tabular}{ccc}
${\bf 2}$ & $\le$ & $2$\\
${\bf 3}$ &  $\le$  & ${4}$ \\
${\bf 5}$ & $\le$ & ${5}$\\
${\bf 7}$ & $\le$ &
\end{tabular}.
\]
\end{example}

The algebra $\mathbb{W}$ for $\msl_n$ is generated by the elements $X_{\tau}^{(l)}$, where $\tau=(\tau_1, \dots, \tau_{l_{\tau}})$,
$1 \leq \tau_1<\dots <\tau_{l{\tau}}\leq n$. We define the following partial order on monomials in $X_{\tau}^{(l)}$.
Let us consider two monomials
\[
u_1=X_{\tau^1}^{(a_1)} \dots X_{\tau^{o}}^{(a_o)},\ u_2=X_{\mu^1}^{(b_1)} \dots X_{\mu^{p}}^{(b_p)},
\]
where $l_{\tau^1}\geq \dots \geq l_{\tau^o}$, $l_{\mu^1}\geq \dots \geq l_{\mu^p}$.
If $p>o$, then $u_2\succ u_1$.
Assume that $p=o$.
If $(l_{\tau^1},\dots , l_{\tau^p})>(l_{\mu^1}, \dots , l_{\mu^p})$ in lexicographic order, then $u_2\succ u_1$.
Assume that $l_{\tau^j}=l_{\mu^j}$, $1 \leq i,j \leq n$.
Let $L=l_{\tau^1}$ and write $\tau^k_j=0$, if $j>l_{\tau^k}$. We consider two sequences:
\[s(u_1)=\left(\sum_{i=1}^p \tau^i_L,\sum_{i=1}^p \tau^i_{L-1}, \dots,\sum_{i=1}^p \tau^i_1\right);\]
\[s(u_2)=\left(\sum_{i=1}^p \mu^i_L,\sum_{i=1}^p \mu^i_{L-1}, \dots,\sum_{i=1}^p \mu^i_1\right).\]
Then if $s(u_2)>s(u_1)$ in usual lexicographic order, then $u_2\succ u_1$.
Assume that $s(u_2)=s(u_1)$. We define the following sequence:
\begin{multline}\label{momomialorder}
sd(u_1)=\Biggl(\sum_{i=1}^p (\tau^i_L-\tau^i_{L-1})^2, \sum_{i=1}^p (\tau^i_{L-1}-\tau^i_{L-2})^2,
\sum_{i=1}^p (\tau^i_L-\tau^i_{L-2})^2, \\
 \sum_{i=1}^p (\tau^i_{L-2}-\tau^i_{L-3})^2,
\sum_{i=1}^p (\tau^i_{L-1}-\tau^i_{L-3})^2,\dots \Biggr)
\end{multline}
and analogously define $sd(u_2)$. Then if $sd(u_2)>sd(u_1)$ in usual lexicographic order, then $u_2\succ u_1$.

\begin{rem}
The order $\succ$ is not sensitive to the upper indices of the variables $X_\tau^{(l)}$. Forgetting these upper indices,
the order  $\succ$ is the total term order.
\end{rem}

Let $\sigma\le \tau$
be two strictly increasing columns of numbers from $1$ to $n$.
\begin{prop}\label{propsnakeplueckerequation}
$a)$.\ For any $k' \leq k(\sigma,\tau)-1$ we have the following equality in $\mathbb{W}[[s]]$:
\begin{equation}\label{snakeplueckerequation}
\sum_{A\subset P(\sigma,\tau), |A|=|\sigma \cap P|}(-1)^{{\rm sign}(\sigma,\tau,P(\sigma,\tau),A)}
\frac{\partial^{k'} X_{\eta_1(\sigma,\tau,P(\sigma,\tau),A)}(s)}{\partial s^{k'}} X_{\eta_2(\sigma,\tau,P(\sigma,\tau),A)}(s)=0.
\end{equation}
$b).$\ We have
$X_{\sigma}^{(a)}X_{\tau}^{(b)}\succeq X_{\eta_1(\sigma,\tau,P(\sigma,\tau),A)}^{(c)}X_{\eta_2(\sigma,\tau,P(\sigma,\tau),A)}^{(d)}$
for any $a,b,c,d$ and the strict inequality holds for all monomials except for the monomials of the form $X_{\sigma}^{(c)}X_{\tau}^{(d)}$.
\end{prop}
\begin{proof}
To prove the first claim we note that equality \eqref{snakeplueckerequation} is a particular case of equality \eqref{Weylpluecker}.

Now let us prove part $b)$.
Assume that $\eta_1(\sigma,\tau,P(\sigma,\tau),A)\neq \sigma$.
Since $P(\sigma,\tau)$ is decreasing, we have
\[s(X_\sigma^{(a)} X_\tau^{(b)})\geq
 s(X_{\eta_1(\sigma,\tau,P(\sigma,\tau),A)}^{(c)} X_{\eta_2(\sigma,\tau,P(\sigma,\tau),A)}^{(d)})\]
 and equality holds if $\eta_1(\sigma,\tau,P(\sigma,\tau),A)_j\in \{\sigma_j,\tau_j\}$,
 $\eta_2(\sigma,\tau,P(\sigma,\tau),A)_j\in \{\sigma_j,\tau_j\}$. Assume that for some $j$
 $\eta_2(\sigma,\tau,P(\sigma,\tau),A)_{j'}=\sigma_{j'}$ for all $j'>j$ and $\eta_2(\sigma,\tau,P(\sigma,\tau),A)_{j}=\tau_{j}$.
 Then $\sigma_j,\tau_j\in P(\sigma,\tau)$. Thus for some $j_1$ either
 \[\sigma_j>\tau_j, \sigma_{j+1}=\tau_{j+1}, \dots, \sigma_{j_1-1}=\tau_{j_1-1},
 \sigma_{j_1}<\tau_{j_1}\]
 or
 \[\sigma_j<\tau_j, \sigma_{j+1}=\tau_{j+1}, \dots, \sigma_{j_1-1}=\tau_{j_1-1},
 \sigma_{j_1}>\tau_{j_1}.\]
 This implies that $sd(X_\sigma^{(a)} X_\tau^{(b)})>
 sd(X_{\eta_1(\sigma,\tau,P(\sigma,\tau),A)}^{(c)} X_{\eta_2(\sigma,\tau,P(\sigma,\tau),A)}^{(d)})$.
\end{proof}

\begin{dfn}\label{barW}
We define the algebra $\overline{\mathbb{W}}$ as the quotient of the polynomial algebra
in variables $X_\tau^{(l)}$ by the ideal generated by relations \eqref{snakeplueckerequation}.
\end{dfn}

We have the following chain of the canonical surjections:

\[\overline{\mathbb{W}}\rightarrow \mathcal{\mathbb{M}}\rightarrow \mathbb{W}.\]

Note that all relations \eqref{snakeplueckerequation} are homogenous with respect to the Cartan weight,
with respect to the $t$-degree (counting the sum of the upper indices of the variables) and with respect to
the homogenous degree.
Consider degenerations of these relations with respect to the partial order $"\succ"$:
\begin{equation}\label{degeneratedequations}
\frac{\partial^{k'}X_{\tau}(s)}{\partial s^{k'}}X_{\sigma}(s)=0
\end{equation}
for any $k' \leq k(\sigma,\tau)-1$. We denote this degenerate algebra by $\widetilde{\mathbb{W}}$.

\begin{prop}\label{characterinequalitydegeneration}$\ch \widetilde{\mathbb{W}} \geq \ch\mathbb{W}.$
\end{prop}
\begin{proof}
Recall the algebra $\overline{\mathbb{W}}$ (Definition \ref{barW}).
We know that $\overline{\mathbb{W}}$ surjects onto $\mathbb{W}$ and hence
$\ch \overline{\mathbb{W}} \geq \ch\mathbb{W}$.

The term order $\succ$ defines the filtration
on the algebra $\overline{\mathbb{W}}$. The associated graded algebra is the quotient of the polynomial algebra
in variables $X_\tau^{(l)}$ by the ideal, which contains all the relations \eqref{degeneratedequations}.
We conclude that $\ch \widetilde{\mathbb{W}} \geq \ch\overline{\mathbb{W}}$.
\end{proof}

\subsection{The character formula}
We compute the character of the degenerated algebra $\widetilde{\mathbb{W}}$.
Recall the notation $(q)_r=\prod_{i=1}^r (1-q^i)$.
For a non-empty subset $\sigma=(\sigma_1,\dots,\sigma_k)\subset\{1,\dots,n\}$ we introduce a variable $r_\sigma$.
\begin{prop}\label{degeneratedcharacter}
\begin{multline}\label{degeneratedcharactereq}
\ch \widetilde{\mathbb{W}}(m_1\omega_1+\dots+m_{n-1}\omega_{n-1})=\\
\sum_{\sum_{|\sigma|=k} r_\sigma=m_k}
\frac{q^{\sum_{\sigma<\tau}k(\sigma,\tau)r_{\sigma}r_{\tau}}\prod_{i=1}^n x_i^{\sum_{\sigma\ni i} r_\sigma}}
{\prod_{\sigma} (q)_{r_\sigma}}.
\end{multline}
\end{prop}
\begin{proof}
We first note that the space $\widetilde{\mathbb{W}}(\la)$ has several gradings. First of all,
attaching degree $l$ to a variable
$X_\sigma^{(l)}$ one gets the standard degree grading. As usual, the power of the variable $q$ takes care of this
grading in the character formula. Now, $\widetilde{\mathbb{W}}(\la)$ is additionally graded by the group
$\bZ_{\ge 0}^{2^n-2}$ with the coordinates
labeled by the proper non-empty subsets $\sigma\subsetneq\{1,\dots,n\}$ (since the relations \eqref{degeneratedequations} are homogeneous).
More concretely, the homogeneous part of degree $\br=(r_\sigma)_\sigma$ is spanned by the monomials in $X_\tau^{(l)}$
such that the number of factors of the form $X_\sigma^{(l)}$ is exactly $r_\sigma$. So it suffices to find the
$q$-character of homogeneous part $\widetilde{\mathbb{W}}(\br)$ of $\widetilde{\mathbb{W}}$ of
degree $\br$. We note that this part
sits in $\widetilde{\mathbb{W}}(\la)$ if and only if $\sum_{|\sigma|=k} r_\sigma=m_k$ for all $k=1,\dots,n-1$.

We consider a functional realization of the dual space of
$\widetilde{\mathbb{W}}(\br)$. Namely, given a linear function $\xi$ on the space
$\widetilde{\mathbb{W}}(\br)$
we attach to it the polynomial $f_\xi$ in variables $Y_{\sigma,j}$, $\sigma\subset \{1,\dots,n\}$, $1\le j\le r_\sigma$ defined as follows.
Recall $X_\sigma(s)=\sum_{l\ge 0} X_\sigma^{(l)}s^l$. Then
\begin{equation}\label{dual}
f_\xi=\xi \bigl(\prod_{\sigma} X_\sigma(Y_{\sigma,1})\dots X_\sigma(Y_{\sigma,r_\sigma})\bigr).
\end{equation}
We claim that formula \eqref{dual} defines an isomorphism between the space of functionals on
$\widetilde{\mathbb{W}}(\br)$
and the space ${\rm Pol}(\br)$ of polynomials $f$ in variables $Y_{\sigma,j}$ subject to the following conditions:
\begin{itemize}
\item $f$ is symmetric in variables $Y_{\sigma,j}$ for each $\sigma$,
\item $f$ is divisible by $(Y_{\sigma,j_1}-Y_{\tau,j_2})^{k(\sigma,\tau)}$ for all $\sigma,\tau,j_1,j_2$.
\end{itemize}
The first condition is obvious and the second one comes from the relations \eqref{degeneratedequations}.
We note that the $q$-degree on $\widetilde{\mathbb{W}}(\br)$ attaching degree $l$ to a variable
$X_\sigma^{(l)}$ is now translated into the counting of the total degree in all variables $Y_{\sigma,j}$.
Now the $q$-character of the space  ${\rm Pol}(\br)$ is given by
\begin{equation}\label{Polch}
\frac{q^{\sum_{\sigma,\tau}k(\sigma,\tau)r_{\sigma}r_{\tau}}\prod_{i=1}^n x_i^{\sum_{\sigma\ni i} r_\sigma}}
{\prod_{\sigma} (q)_{r_\sigma}}.
\end{equation}
Indeed, $(q)_r^{-1}$ is the character of the space of symmetric polynomials in $r$ variables and the factor
$\prod_{\sigma<\tau}(Y_{\sigma,j_1}-Y_{\tau,j_2})^{k(\sigma,\tau)}$ produces the numerator of the above formula.
\end{proof}

\begin{cor}\label{corssst}
The homogeneous component $\widetilde{\mathbb{W}}(\br)\subset \widetilde{\mathbb{W}}$ has the basis consisting of monomials of the form
\begin{equation}\label{ssst}
\prod_{\sigma\subset\{1,\dots,2n\}} X_\sigma^{(l_{1,\sigma})}\dots X_\sigma^{(l_{r_\sigma,\sigma})},\quad 0\le l_{1,\sigma}\le \dots\le l_{r_\sigma,\sigma}
\end{equation}
such that $l_{1,\tau}\ge \sum_{\sigma<\tau} k(\sigma,\tau) r_\sigma$ for all $\tau$.
\end{cor}
\begin{proof}
We note that the character of the set of monomials \eqref{ssst} is equal to \eqref{Polch}. Hence it suffices to show that the elements
\eqref{ssst} span the space $\widetilde{\mathbb{W}}(\br)$.

Assume that there exists an element $\xi\in (\widetilde{\mathbb{W}}(\br))^*$ vanishing on all the monomials \eqref{ssst}. We want to show that
in this case $\xi$ is zero; equivalently, we need to prove that $f_\xi=0$.
A non-zero polynomial divisible by $\prod_{\sigma<\tau}(Y_{\sigma,j_1}-Y_{\tau,j_2})^{k(\sigma,\tau)}$ contains a monomial
\begin{equation}\label{Ymon}
\prod_{\sigma\subset\{1,\dots,2n\}} Y_\sigma^{l_{1,\sigma}}\dots Y_\sigma^{l_{r_\sigma,\sigma}},\quad 0\le l_{1,\sigma}\le \dots\le l_{r_\sigma,\sigma}
\end{equation}
such that $l_{1,\sigma}\ge k(\sigma,\tau)r_\tau$ for any pair $\sigma<\tau$ (coming from the choice of the term $Y_{\sigma,j_1}^{k(\sigma,\tau)}$ in each bracket
$(Y_{\sigma,j_1}-Y_{\tau,j_2})^{k(\sigma,\tau)}$). However, the coefficient in front of monomial \eqref{Ymon} is equal to
$\xi(X_\sigma^{(l_{1,\sigma})}\dots X_\sigma^{(l_{r_\sigma,\sigma})})$. Therefore, if $\xi$ vanishes on all the elements \eqref{ssst}, then
$f_\xi$ is zero.
\end{proof}

\begin{example}
Consider the case of Lie algebra $\fg\simeq \msl_3$. Then the algebra $\widetilde{\mathbb{W}}$
is generated by the elements $X_1^{(l)}$, $X_2^{(l)}$, $X_3^{(l)}$,
$X_{12}^{(l)}$, $X_{13}^{(l)}$, $X_{23}^{(l)}$, $l=0,1,\dots$ subject to the relations
\[
\sum_{l_1+l_2=N} X_1^{(l_1)}X_{23}^{(l_2)} - X_2^{(l_1)}X_{13}^{(l_2)} + X_3^{(l_1)}X_{12}^{(l_2)}
\]
for all $N\ge 0$.
 Our basis consists of elements of the following form:
\[\prod_{j=1}^{r_1}X_1^{(l_{j,1})}\prod_{j=1}^{r_2}X_2^{(l_{j,2})}\prod_{j=1}^{r_3}X_3^{(l_{j,3})}
\prod_{j=1}^{r_{12}}X_{12}^{(l_{j,12})}\prod_{j=1}^{r_{13}}X_{13}^{(l_{j,13})}\prod_{j=1}^{r_{23}}X_{23}^{(l_{j,23})},\]
$l_{j,\sigma}\leq l_{j',\sigma}$ for $j<j'$ and $l_{j,23} \geq r_1$ for all $j$.
\end{example}

\begin{cor}
In Section \ref{Weyl} we show that the characters of the algebras ${\mathbb{W}}$ and its degenerate version  $\widetilde{\mathbb{W}}$ coincide.
This implies that the union of monomials \eqref{ssst} over all $\br$ such that $\sum_{|\sigma|=k} r_\sigma=m_k$ for all $k=1,\dots,n-1$
is a basis of $\mathbb{W}(\lambda)^*$.
\end{cor}

\begin{rem}
The $t$-degree zero part of the basis \eqref{ssst} is given by the semi-standard tableaux. In fact, one needs the monomials \eqref{ssst}
with zero $l_{a,\sigma}$ for all $a$ and $\sigma$. This is possible if and only if for any two distinct $\sigma$, $\tau$ such that $r_\sigma>0$ and $r_\tau>0$ one
has $k(\sigma,\tau)=0$. In other words, the indices $\sigma$ of a degree zero basis vector can be packed into the semi-standard tableau.
\end{rem}

\begin{rem}
In contrast to the degenerate algebra $\widetilde{\mathbb{W}}$, the algebra $\mathbb{W}$ does not
have that many gradings. Instead of the group $\bZ_{\ge 0}^{2^n-2}$ the nondegenerate algebra has only Cartan
$\bZ_{\ge 0}^{n-1}$ grading.
\end{rem}

Let us introduce the notation for the $q$-multinomial coefficient
\[
\bin{m_k}{[\br]_k}_q=\frac{(q)_{m_k}}{\prod_{|\sigma|=k} (q)_{r_\sigma}}.
\]
Recall $(q)_\la=\prod_{k=1}^{n-1} (q)_{m_k}$.
\begin{cor}\label{characterineq}
\begin{equation}\label{ineq}
(q)_{\lambda}^{-1}\sum_{\sum_{|\sigma|=k} r_\sigma=m_k}
q^{\sum_{\sigma<\tau}k(\sigma,\tau)r_{\sigma}r_{\tau}}\prod_{i=1}^n x_i^{\sum_{\sigma\ni i} r_\sigma}
\prod_{k=1}^{n-1} \bin{m_k}{[\br]_k}_q\geq {\rm ch}\mathbb{W}(\lambda).
\end{equation}
\end{cor}
\begin{proof}
This is a reformulation of inequality \eqref{degeneratedcharactereq} using the definition
of $q$-multinomial coefficient.
\end{proof}

Recall the relation between the characters of the local and global Weyl modules:
\begin{equation}\label{global-local}
\ch \mathbb{W}(\la)=(q)_\la^{-1}\ch W(\la).
\end{equation}
Let $W(\la)=\bigoplus_{l\ge 0} W(\la)^{(l)}$ be the $q$-degree decomposition of the local Weyl module.
In particular, the character of each $W(\la)^{(l)}$ is a polynomial in $x_1,\dots,x_n$.
Let $\sum_{l\ge 0} q^l C_l(x_1,\dots,x_n)$ be the $q$-expansion of the left hand side of \eqref{ineq} multiplied by $(q)_{\lambda}$.
In particular, each $C_l$ is a Laurent polynomial in $x_i$.
\begin{lem}\label{ineqchar}
Let $j$ be the smallest number such that $\ch W(\lambda)^{(j)}\ne C_j$. Then
$\ch W(\lambda)^{(j)}< C_j$ coefficient-wise.
\end{lem}
\begin{proof}
Note that the claim is true for global Weyl modules instead of local. Thus the proposition follows from relation
\eqref{global-local} between the characters.
\end{proof}


\section{Evaluation modules}\label{Weyl}
In this section we consider representations of the Lie algebras $\fg=\msl_n$ and $\fg=\msl_{2n}$.
\subsection{Fusion construction}
For any $k=1,\dots,n-1$ let $v_{\om_k}\in V(\omega_k)$ be a highest weight vector.
Given $\zeta \in \mathbb{K}$  and $k=1,\dots,n-1$ we
consider the evaluation $\fg\T\mathbb{K}[t]$-module $V(\omega_k)_\zeta$, which is the cyclic module with cyclic vector
$v_{\om_k}$ and the following action of current algebra:
\begin{equation}\label{evaluationdefinition}
x \otimes t^a.v=\zeta^a xv, x \in \fg, v \in V(\omega_k).
\end{equation}

Let $\lambda=\sum_{k=1}^{n-1} m_k \omega_k$ and let $(\zeta_{k,i})$, $1 \leq k \leq n-1$, $1 \leq i \leq m_k$
be a tuple of pairwise distinct elements. We consider the tensor product:
\[V(\lambda)_{(\zeta_{k,i})}=\bigotimes_{k=1}^{n-1}\bigotimes_{i=1}^{m_k}V(\omega_k)_{\zeta_{k,i}}.\]

\begin{prop}\cite{CL,FL1,FL2,N}\label{fusionproduct}
The module $V(\lambda)_{(\zeta_{k,i})}$ is cyclic with cyclic vector $v$ of weight $\lambda$. The $t$-degree grading on
$U(\fg \otimes \mathbb{K}[t])$ gives a filtration on $V(\lambda)_{(\zeta_{k,i})}$:
\begin{equation}\label{filtrationevaluation}
\{0\} =F_{-1} \subset U(\fg \otimes 1)v= F_0\subset F_1\subset F_2 \subset \dots
\end{equation}
The corresponding graded module $\bigoplus_{i=0}^\infty F_{i}/F_{i-1}$ is isomorphic to $W(\lambda)$.
\end{prop}

Assume that we have a basis $\lbrace f_{\gamma_{o,1}}\otimes t^{l_{o,1}}\dots f_{\gamma_{o,q_o}}\otimes t^{l_{o,q_o}}v\rbrace$
of the space of elements of weight $\mu$ in $V(\lambda)_{(\zeta_{k,i})}$, $\gamma_{o,j}\in \Delta_-$,
$o$ runs from 1 to the dimension of this space. Its character is a polynomial
\[
\sum_o q^{\sum_{j=1}^{q_o}l_{o,j}}=\sum_{i \geq 0} a_{\mu,i}q^i.
\]
Let $\sum_{i \geq 0} b_{\mu,i}q^i$ be the character of the space of elements of weight $\mu$ in $W(\lambda)$.
\begin{lem}\label{inverceineqchar}
Let $j$ be the smallest number such that $a_{\mu,j}\neq b_{\mu,j}$. Then $a_{\mu,j}< b_{\mu,j}.$
\end{lem}
\begin{proof}
An element
\begin{equation}\label{dif}
f_{\gamma_{o,1}}\otimes t^{l_{o,1}}\dots f_{\gamma_{o,q_o}} \otimes t^{l_{o,q_o}}v
\end{equation}
belongs to $F_{\sum_{i=1}^{q_o} {l_{o,i}}}$.
Then the $t$-degree of the image of this element in the adjoint graded space (aka Weyl module) is less than or equal to
$\sum_{i=1}^{q_o} {l_{o,i}}$.
If $j$ is the smallest number such that $a_{\mu,j}\neq b_{\mu,j}$, then all the elements \eqref{dif} with
$\sum_{i=1}^{q_o} {l_{o,i}}=j$ do not belong to the filtration space $F_{j'}$ for $j'<j$.
In addition, there exists a basis element \eqref{dif} with $\sum_{i=1}^{q_o} {l_{o,i}}>j$
which belongs to $F_j$. Hence $b_{\mu,j}> a_{\mu,j}.$
\end{proof}

\begin{cor}
If there exists a basis of the tensor product $V(\lambda)_{(\zeta_{k,i})}$ such that its character is equal to the
left hand side of \eqref{ineq} multiplied by $(q)_\la$, then the inequality \eqref{ineq} is in fact an equality.
\end{cor}
\begin{proof}
The Corollary is implied by Proposition \ref{localglobal}, Lemma \ref{ineqchar} and Lemma \ref{inverceineqchar}.
\end{proof}
In the rest of the section we construct a basis with the desired property. The construction uses an embedding
of a Weyl module attached to $\msl_n$ into a Weyl module  attached to $\msl_{2n}$.

\subsection{Embedding of Weyl modules}
Let $\fg=\msl_{2n}$ and let $v_1,\dots,v_{2n}$ be the standard basis of the vector representation of $\msl_{2n}$.
We denote the weight of $v_j$ by $\epsilon_j$.
For a cardinality $k$ set $J=(j_1<\dots<j_k) \subset \{1,\dots,2n\}$ we denote by $v_J\in V(\omega_k)$ the wedge product
$v_{j_1}\wedge\dots\wedge v_{j_k}$. We also denote by $f_{pq}\in\msl_{2n}$, $p<q$ the matrix unit sending $v_p$ to $v_q$.

We consider the inclusion $\msl_n\subset \msl_{2n}$ via tautological map $e_{ij} \mapsto e_{ij}$,
$1 \leq i,j \leq n$, where the first matrix unit is an element of $\msl_n$ and the second one is an element of $\msl_{2n}$.
Let $\lambda=\sum_{k=1}^{n-1}m_k\omega_k$ be a dominant weight; in what follows we consider $\la$ as both
$\msl_n$ and $\msl_{2n}$ weight. To distinguish these cases we write $\overline{\la}$ in the $\msl_{2n}$ case.

Let $v$ be a cyclic vector of the $\msl_{2n}\T \mathbb{K}[t]$ Weyl module $W(\overline{\la})$.
The following proposition can be extracted from the results of \cite{CL}, but for the readers convenience
we give a short proof below.
\begin{prop}
The $\msl_n\T \mathbb{K}[t]$-module  $U(\msl_n)v\subset W(\overline{\la})$ is isomorphic to the Weyl module $W(\lambda)$.
It is equal
to the set of elements of $W(\overline{\la})$ with weights in $\mathbb{Z}\langle \epsilon_1, \dots, \epsilon_n \rangle$.
\end{prop}
\begin{proof}
Note that all relations of the Weyl module $W(\lambda)$ hold in $U(\msl_n)v$, so we only need to compare the dimensions.
The  Weyl module $W(\overline{\la})$ is generated from $v$ with the action of the operators $f_{pq}\T t^l$.
We note that a vector
$\{\prod_{l=1}^d f_{i_lj_l}\otimes t^{x_l}.v\}$ with some $j_a > n$ never has a weight in
$\mathbb{Z}\langle \epsilon_1, \dots, \epsilon_n \rangle$.
However the dimension of the space of  elements of weight in $\mathbb{Z}\langle \epsilon_1, \dots, \epsilon_n \rangle$
is exactly equal to $\dim W(\lambda)$
because of Proposition \ref{fusionproduct}.
\end{proof}

\subsection{Combinatorial construction}
For an $\msl_{2n}$-weight $\lambda=\sum_{k=1}^{n-1} m_k \omega_k$ we consider a tuple of pairwise distinct elements
$(\zeta_{k,i})$, $k=1, \dots, n-1$, $1 \leq i \leq m_k$ and the tensor product:
\[V(\lambda)_{(\zeta_{k,i})}=\bigotimes_{k=1}^{n-1}\bigotimes_{i=1}^{m_k}V(\omega_k)_{\zeta_{k,i}}.\]
We construct a basis of this module.

Let $\mathbb{B}$ be the set consisting of collections $(\mathcal{B}_{k,i})$ labeled by pairs $k=1, \dots, n-1$,
$1 \leq i \leq m_k$,
where each $\mathcal{B}_{k,i}$ is the set consisting of elements $f_{pq}$, $1\le p<q\le n$ satisfying two following conditions:

${\bf F1}.$ If $f_{pq}\in \mathcal{B}_{k,i}$, then $p \leq k<q$.

${\bf F2}.$ If $f_{p_1q_1}\in \mathcal{B}_{k,i}$ and $f_{p_2q_2} \in \mathcal{B}_{k,i}$,
then either $p_1<p_2$ and $q_2<q_1$ or $p_1>p_2$ and $q_2>q_1$.

\begin{example}
For $n=5$ the set $\mathbb{B}$ contains the element $B=(\mathcal{B}_{2,1},\mathcal{B}_{3,1})$ defined by
$\mathcal{B}_{2,1}=\{f_{1,9}, f_{2,8}\}$, $\mathcal{B}_{3,1}=\{f_{1,10}, f_{2,7}, f_{3,6}\}$.
\end{example}
\begin{lem}\label{oneplacefilling}
For any pair $k,i$ the sets $\mathcal{B}_{k,i}$ satisfying $F1$ and $F2$ are in natural bijection with certain basis of $V(\omega_k)$.
\end{lem}
\begin{proof}
The basis of $V(\omega_k)$ we need consists of vectors
\[
\prod_{j=1}^l f_{p_jq_j}v=v_{\{1, \dots,k\}\backslash\{p_1,\dots, p_l\}\cup\{q_1,\dots, q_l\}},
\]
where $p_j>p_{j'}$, $q_j<q_{j'}$ for $j<j'$ (see e.g. \cite{FFL1}).
\end{proof}

 For a pair $(k,i)$ let $v_{k,i}$ be a highest weight vector of the module $V(\omega_k)_{\zeta_{k,i}}$.

\begin{cor}
We have a bijection $\mathcal{P}$ between the set $\mathbb{B}$ and the basis of $V(\lambda)_{(\zeta_{k,i})}$ defined by the formula
\begin{multline*}\mathcal{P} \left((\mathcal{B}_{k,i})\right)=
\left(\prod_{f_{p,q}\in \mathcal{B}_{1,1}} f_{pq}\right)v_{1,1}\otimes
\left(\prod_{f_{p,q}\in \mathcal{B}_{1,2}} f_{pq}\right)v_{1,2}\otimes \dots\\ \otimes
\left(\prod_{f_{p,q}\in \mathcal{B}_{n-1,m_{n-1}}} f_{pq}\right)v_{n-1,m_{n-1}} .\end{multline*}
\end{cor}

Note that all $f_{p,q}$ showing up in $\mathcal{B}_{k,i}$ commute because of condition $F1$ Therefore this basis is well defined.
To a set $\mathcal{B}_{k,i}$ we attach the set $J_{k,i}\subset \{1,\dots,2n\}$ defined by
\begin{equation}\label{BJ}
\prod_{f_{pq}\in\mathcal{B}_{k,i}} f_{pq}v_{\omega_k}= v_{J_{k,i}}=\pm \bigwedge_{j\in J_{k,i}} v_j.
\end{equation}
Then
\[
\mathcal{P} \left((\mathcal{B}_{k,i})\right)=\pm \bigotimes_{\substack{1\le k\le n-1\\ 1\le i\le m_k}} v_{J_{k,i}}.
\]



We consider the following 
order on the elements $f_{pq}$, $1\le p<q\le 2n$:
\begin{equation}\label{forder}
f_{12}<f_{13}<f_{23}<f_{14}<f_{24}<f_{34}<f_{15}< \dots
\end{equation}
We fix an element $B=(\mathcal{B}_{k,i})\in \mathbb{B}$.
For any element $f_{pq}\in \mathcal{B}_{k,i}$  we attach a degree in the following way.

First, we consider the restriction of $B$ to
elements less than $f_{pq}$, i.e $B'=(\mathcal{B}'_{k,i})$, where
$\mathcal{B}'_{k,i}=\mathcal{B}_{k,i}\cap\{f_{ab}, f_{ab}<f_{pq}\}$.

Second, we consider all pairs $(k,i)$ such that $\mathcal{B}'_{k,i}\cup \{f_{pq}\}$ satisfies conditions $F1$ and $F2$.
We call these pairs $(p,q)$-admissible.

Third, we define the following order on $(p,q)$-admissible pairs.  Let $(k_1,i_1)$ and $(k_2,i_2)$ be two $(p,q)$-admissible pairs. Let
\[
\mathcal{B}'_{k_1,i_1}=\{f_{p_{1,j}q_{1,j}}, 1 \leq j \leq l_1\},\
\mathcal{B}'_{k_2,i_2}=\{f_{p_{2,j}q_{2,j}}, 1 \leq j \leq l_2\}.
\]
If $k_1-l_1<k_2-l_2$, then $(k_1,i_1)<(k_2,i_2)$.
Assume that $k_1-l_1=k_2-l_2$. We consider vectors
\[\bar q_1=(q_{1,l_1},\dots,q_{1,1}) \text{ and } \bar q_2=(q_{2,l_2},\dots,q_{2,1}).
\]
If $\bar q_1 <\bar q_2$ in lexicographic order (comparing from left to right, i.e. we first compare
$q_{1,l_1}$ with $q_{2,l_2}$, then $q_{1,l_1-1}$ with $q_{2,l_2-1}$, etc.), then $(k_1,i_1)<(k_2,i_2)$.
Finally if $k_1-l_1=k_2-l_2$, $\bar q_1 =\bar q_2$ (therefore $k_1=k_2$) and $i_1<i_2$, then
$(k_1,i_1)<(k_2,i_2)$.

Fourth, we consider all $(p,q)$-admissible pairs with the above order. For an element
$f_{p,q}\in \mathcal{B}_{k,i}$
we attach $t$-degree $d(p,q,k,i)$ according to  the following definition.

\begin{dfn}\label{dpqki}
$d(p,q,k,i)$ is equal to the number of all $(p,q)$-admissible
pairs $(k',i')$ such that $(k',i')<(k,i)$ and $f_{pq}\notin \mathcal{B}_{k',i'}$.
\end{dfn}

Finally, we define the following element of the universal enveloping algebra $U(\msl_{2n}\otimes \mathbb{K}[t])$:
\begin{equation}
\Pi(B)=\prod_{f_{pq}\in\mathcal{B}_{k,i}} f_{pq}\otimes t^{d(p,q,k,i)}.
\end{equation}
The product in the formula above is taken in such a way that smaller $f_{pq}$ are applied first
(i.e. the smaller $f_{pq}$ show up on the right).

\begin{dfn}
We define $\mathcal{F}(B)\in V(\lambda)_{(\zeta_{k,i})}$ as $\Pi(B)v$.
\end{dfn}

\subsection{Basis}
Let us write a weight over $\msl_{2n}$ as a linear combination of $\epsilon_i$, $i=1, \dots, 2n$; a root vector $f_{pq}$
has the weight $\epsilon_q-\epsilon_p$ and the highest weight of a fundamental representation $V(\omega_k)$
is equal to $\sum_{i=1}^k\epsilon_i$.

Let $\mathbb{B}_{\mu}$ be the set of elements $B \in \mathbb{B}$ such that the weight of $\mathcal{F}(B)$ is equal to $\mu$
(equivalently the weight of $\mathcal{P}(B)$ is equal to $\mu$).

\begin{rem}\label{notall}
In what follows we only consider weights $\mu\in \bZ_{\ge 0}\langle \epsilon_{n+1}, \dots, \epsilon_{2n} \rangle$.
For such a weight let $B$ be an element in $\mathbb{B}_{\mu}$. Then conditions $F1$ and $F2$ imply that for each
pair $(k,i)$ the set $\mathcal{B}_{k,i}$
consists of elements $f_{p,q}$ satisfying the following conditions: $1\le p<n<q\le 2n$ and $p+q\le 2n+1$.
\end{rem}

\begin{dfn}
A subset $J\subset\{1,\dots,2n\}$ is called dense if $J\cap \{1,\dots,n\}=\{1,\dots,a\}$ for some $a$.
An element $B \in \mathbb{B}$ is called dense, if all entries $\mathcal{B}_{k,i}$ correspond to dense sets $J_{k,i}$ in the
sense of \eqref{BJ}. A vector $\mathcal{P}(B)$ is called dense if $B$ is dense.
\end{dfn}
\begin{rem}\label{densecond}
We note that $B$ is not dense if and only if there exist $a,b,k,i$, $a<k$ such that
$f_{ab}\in {\mathcal{B}}_{k,i}$ and there is no element $c$ such that
$f_{a+1,c}\in {\mathcal{B}}_{k,i}$.
\end{rem}

Now let us take any $f_{pq}$ and put $\mathcal{B}_{k,i}'=\mathcal{B}_{k,i}\cap \{f_{ab}, f_{ab}<f_{pq}\}$,
$B'=(\mathcal{B}_{k,i}')$.
We call $B'$ a strict $(p,q)$-restriction of $B$. In the following lemma we fix $p,q$ with $1\le p<n<q\le 2n$.

\begin{lem}\label{dense}
For a $\msl_{2n}$-weight $\mu\in \bZ_{\ge 0}\langle \epsilon_{n+1}, \dots, \epsilon_{2n} \rangle$ and
any $B \in \mathbb{B}_\mu$ its strict $(p,q)$-restriction $B'$ is dense.
\end{lem}
\begin{proof}
Follows from the explicit form of the order \eqref{forder} and property $F2$.
\end{proof}

Recall that we consider irreducible highest weight representations $V(\lambda)$ of $\msl_{2n}$ with
$\la=\sum_{k=1}^{n-1} m_k\omega_k$.

\begin{prop}\label{basisevaluation}
For an $\msl_{2n}$ weight $\mu\in \bZ_{\ge 0}\langle \epsilon_{n+1}, \dots, \epsilon_{2n} \rangle$
the elements $\{\mathcal{F}(B)\}$, $B \in  \mathbb{B}_{\mu}$ constitute a basis of
the $\mu$-weight space of
\[V(\lambda)_{(\zeta_{k,i})}=\bigotimes_{k=1}^{n-1}\bigotimes_{i=1}^{m_k}V(\omega_k)_{\zeta_{k,i}}.\]
\end{prop}
\begin{proof}
We divide the proof into four steps. In {\it Step 1} we consider the decomposition of an element $\mathcal{F}(B)$
in the basis $\mathcal{P}(\mathbb{B}_\mu)$. In {\it Step 2} we restrict to the "top" summands from the decomposition
of {\it Step 1}. In {\it Step 3} we further restrict to the dense summands of the "top" part from {\it Step 2}.
In {\it Step 4} we finalize the proof by proving the non-degeneracy of the transition matrix from the set
$\mathcal{F}(\mathbb{B}_\mu)$ to the dense terms from {\it Step 3}.

{\it Step 1}. Recall that we have order \eqref{forder} on elements $f_{pq}$.
We consider the lexicographic order on sequences $(f_{p_{j}q_{j}}, j=1, \dots, a)$ with
$f_{p_{j}q_{j}}\geq f_{p_{j+1}q_{j+1}}$. In other words, we first compare the largest elements
of two sequences, then the next to the largest elements, etc.
This order gives an order on multisets of elements $f_{pq}$.

Note that the weight of an element $\mathcal{P}\left((\mathcal{B}_{k,i})\right)$ depends only on the multiset
$M(B)=\sqcup_{k,i}\mathcal{B}_{k,i}$ (i.e. each appearance of an $f_{p,q}$ in a $\mathcal{B}_{k,i}$
increases the multiplicity of $f_{p,q}$ in $M(B)$). For $B \in \mathbb{B}_{\mu}$ with
$\mu\in \bZ_{\ge 0}\langle \epsilon_{n+1}, \dots, \epsilon_{2n} \rangle$, the multiset $M(B)$ consists
of elements $f_{pq}$ with the property $p<n<q$ (recall that $\la\in \bZ_{\ge 0}\langle \epsilon_1,\dots,\epsilon_n\rangle$).
We denote the weight of $\mathcal{P}\left((\mathcal{B}_{k,i})\right)$
by $wt(M(B))$. Note that if $wt(M(B))\in \bZ_{\ge 0}\langle \epsilon_{n+1}, \dots, \ \epsilon_{2n}\rangle$, then
$|M(B)|=\sum_{k=1}^{n-1} km_k$.

We first claim that $\mathcal{F}(B)\in \bigoplus_{M(\tilde B)\leq M(B)}\mathbb{K}\mathcal{P}(\tilde B)$.
Indeed, take any summand of $\mathcal{F}(B)$ in the basis $\mathcal{P}(\mathbb{B})$. It is of the form
\[
\bigotimes_{\substack{1\leq k \leq n-1\\ 1 \leq i \leq m_k}}\prod_{j=1}^{l_{i,k}}f_{p_{j,i,k}q_{j,i,k}}v_{\omega_k}=
\mathcal{P}(\tilde B)
\]
 and the multiset $\{f_{p_{j,i,k}}\}$
is equal to $M(B)$. Assume that
\[
\prod_{j=1}^{l_{i,k}}f_{p_{j,i,k}q_{j,i,k}}v_{\omega_k}=\pm\prod_{j=1}^{l_{i,k}}f_{p'_{j,i,k}q'_{j,i,k}}v_{\omega_k}
\]
and the set $\{f_{p'_{j,i,k}q'_{j,i,k}}, j=1, \dots, l_{i,k}\}$ satisfies condition $F2$. Then it is easy to see that
$\{f_{p_{j,i,k}q_{j,i,k}}, j=1, \dots, l_{i,k}\}\geq\{f_{p'_{j,i,k}q'_{j,i,k}}, j=1, \dots, l_{i,k}\}$.
Therefore $M(\tilde B)<M(B)$.

{\it Step 2}.
Let us fix a multiset $M$. We prove that the classes of elements $\mathcal{F}(B)$, $M(B)=M$, form a basis of
$\bigoplus_{M(B)\leq M}\mathbb{K}\mathcal{P}(B)/\bigoplus_{M(B)< M}\mathbb{K}\mathcal{P}(B)$.

We fix a pair $(p,q)$ and consider the set $\mathcal{B}_{k,i}'=\mathcal{B}_{k,i}\cap \{f_{ab}, f_{ab}<f_{pq}\}$, $B'=(\mathcal{B}_{k,i}')$.
Consider the decomposition of $\mathcal{F}(B')$ in the basis $\mathcal{P}(\mathbb{B})$. Take a non-dense summand of this
decomposition of
the form $\mathcal{P}(\widetilde B)$, $\widetilde B=(\widetilde{\mathcal{B}}_{k,i})$
(see Definition \ref{dense} and Remark \ref{densecond}).
Then we have:
\begin{equation}\label{M(B)}
\left(\prod_{\substack{f_{p'q'}\in\mathcal{B}_{k,i}\\ f_{p'q'} \geq f_{pq}}}
f_{p'q'}\otimes t^{d(p,q,k,i)}\right)\mathcal{P}(\widetilde{B})
\in \bigoplus_{\substack{M(B)< M\\ B\in \mathbb{B}_\mu}}\mathbb{K}\mathcal{P}(B).
\end{equation}
Indeed, any nonzero summand of the decomposition of the left hand side of \eqref{M(B)} in the basis $\mathcal{P}(\mathbb{B})$
has $(k,i)$-th tensor factor of the form $f_{ab}f_{a+1,b'}\prod_{xy}{f_{xy}}v_{\omega_k}$, $b'>b$
(since $f_{a+1,b'}>f_{a,b}$). Therefore the set $\{f_{ab}f_{a+1,b'}\}\cup \{f_{xy}\}$  does not satisfy
condition $F2$, which implies $M(B)<M$.

{\it Step 3.}
Recall that we have fixed a multiset $M$ and a pair of indices $p,q$ with $1\le p<n<q\le 2n$.
Let $\mathbb{B}_\mu^{<pq}$ be the set of strict $(p,q)$-restrictions of the elements $B\in \mathbb{B}_\mu$ with $M(B)=M$.
We denote the elements of $\mathbb{B}_\mu^{<pq}$ by $B'_1,\dots,B'_g$, so there exist elements
$B_1,\dots,B_g\in \mathbb{B}_\mu$, $M(B_j)=M$ such that the strict $(p,q)$-restriction of $B_j$ is $B_j'$.
For an element $B'_j$, $1\le j\le g$ we consider the decomposition $\mathcal{F}(B'_j)=X+Y$ in the basis $\mathcal{P}(\mathbb{B})$,
where $X$ is the linear combination of non-dense summands and $Y$ is the linear combination of dense summands.
We claim that
\begin{equation}\label{Y}
Y=\sum_{l=1}^g a_{jl}\mathcal{P}(B'_l)
\end{equation}
and the matrix $(a_{jl})$ is non-degenerate.
This  claim would imply Proposition \ref{basisevaluation} thanks to Step 2 above.
We note that Proposition \ref{basisevaluation} follows from the non-degeneracy of the unrestricted matrix $(a_{jl})$.
We introduce the $pq$-restriction in order to be able to use the induction on $f_{pq}$.

We first explain the existence of the decomposition \eqref{Y}.
Let $(B_j)_{k,i}=(B_j')_{k,i}\sqcup (\bar B_j)_{k,i}$, so $(\bar B_j)_{k,i}$ consists of $f_{a,b}$ such that
$f_{a,b}\ge f_{p,q}$.
Let $\mathcal{P}(B')$ be a dense summand of $Y$.
It suffices to show that there exists a permutation $\sigma$ of the set $\{(k,i): 1\le k\le n-1, 1\le i\le m_k\}$
such that for each pair $k,i$ the set $B'_{k,i}\sqcup (B_j)_{\sigma(k,i)}$ satisfies conditions $F1$ and $F2$.
Such a $\sigma$ is constructed as follows: let $a=1,\dots,n-1$ be the minimal number such that $f_{a,\bullet}\in B'_{k,i}$.
We fix a pair $k',i'$ such that $a$ is the minimal number with $f_{a,\bullet}\in (B_j)_{k',i'}$. Then
$\sigma(k,i)=(k',i')$. We note that such a pair $(k',i')$ may not be unique, however, there is one for each
$(k,i)$ with $a$ as above.

{\it Step 4}. In the rest of the proof we show by induction on $(p,q)$ that the matrix $(a_{jl})$ is non-degenerate.
The base of induction $(p,q)=(1,n+1)$ is trivial.
Assume that $(a_{jl})$ is non-degenerate for the strict $(p,q)$-restriction.

Several observations are in order.
First, for a (dense) summand $\mathcal{P}(B_l')$, $B_l'=((\mathcal{B}_l)_{k,i}')$ there are exactly
\[
L=|\{(k,i),f_{p+1,q'}\in \mathcal{B}_{k,i},q'<q\}|-|\{(k,i),f_{p,\tilde q}\in \mathcal{B}_{k,i},\tilde q<q\}|
\]
pairs $k,i$ such that $(\mathcal{B}'_l)_{k,i}\cup\{f_{pq}\}$ is still dense.
It is important to note that $L$ does not change when we vary $l$, i.e. $L$ is completely determined by $M$.

Second, the same number $L$ is the cardinality of the set $S_l$ of pairs $(k,i)$ such that one can add $f_{p,q}$
to $(\mathcal{B}'_l)_{k,i}$ with $(k,i)$ from $S_l$
in such a way that the result is still a restriction of an element $B\in \mathbb{B}_\mu$ with $M(B)=M$.
Moreover, the set $S_l$ coincides with the set of pairs $(k,i)$ from the first observation above.

Third, we note that all pairs $(k,i)\in S_l$ are $(p,q)$-admissible. We observe that if $(k,i)\in S_l$ and
$(k',i')\notin S_l$ are two $(p,q)$-admissible positions, then $(k,i)<(k',i')$ (since $(k',i')\notin S_l$
means that the difference between $k'$ and the cardinality of $(\mathcal{B}'_l)_{k',i'}$ is greater than $p$,
while the same difference for the pair $(k,i)$ is equal to $p$).

Fourth, we fix $l=1,\dots,g$ and consider the tensor product of $L$ copies of evaluation fundamental $\msl_2$ modules
$\bigotimes_{(k,i)\in S(l)} V(\omega)_{\zeta_{k,i}}$.
The role of $f\in\msl_2$ is played by the element $f_{p,q}$. The set $\mathbb{B}$ in this case has $2^L$ elements
(recall that $L=|S_l|$) and the transition matrix $c_{u,u'}(l)$ from $\mathcal{F}(\mathbb{B})$ to $\mathcal{P}(\mathbb{B})$
is non-degenerate. We conclude that the transition matrix from the $\mathcal{F}$-basis to the $\mathcal{P}$-basis after
$f_{p,q}$ is added (i.e. for $(p+1,q)$-restriction if $p\le n-1$ of for the $(n+1,q+1)$-restriction, if $p=n-1$)
is equal to $a_{jl}c_{u,u'}(l)$. The determinant of this matrix is equal to $\det (a_{j,l})^L\prod_l \det (c_{u,u'}(l))$.
This completes the proof of the Proposition.
\end{proof}

\subsection{Generating function}
Let $w_0$ be a permutation interchanging $i$ and $2n+1-i$
and let
$\mathbb{B}_{>n}=\cup_{\mu\in \mathbb{N}\langle \epsilon_{n+1}, \dots, \epsilon_{2n} \rangle}\mathbb{B}_{\mu}$.
We also fix the notation $wt(f_{a,c})=\epsilon_c-\epsilon_a$ and
\[
wt(B)=\sum_{\substack{k=1,\dots,n-1,\\ i=1,\dots,m_k}} \sum_{f_{a,c}\in\mathcal{B}_{k,i}} wt(f_{ac}).
\]
In particular, $wt B$ is always non positive (negative if $B$ is nonempty).
In the rest of the section we prove the following equality:
\begin{multline}\label{baschar}
\sum_{B \in \mathbb{B}_{>n}}q^{\sum_{k,i}\sum_{f_{ac}\in\mathcal{B}_{k,i}}d(a,c,k,i)}x^{w_0(\la+wt(B))}\\=
\sum_{\sum_{|\sigma|=k} r_\sigma=m_k}
q^{\sum_{\sigma<\tau}k(\sigma,\tau)r_{\sigma}r_{\tau}}\prod_{i=1}^n x_i^{\sum_{\sigma\ni i} r_\sigma}
\prod_{k=1}^{n-1} \bin{m_k}{[\br]_k}_q
\end{multline}
(to be compared with \eqref{characterineq}).

Recall that we are working with fundamental representations of the Lie algebra $\msl_{2n}$. More precisely, we are
only interested in the weight spaces corresponding to the weights being linear combinations of
$\epsilon_{n+1},\dots,\epsilon_{2n}$. In a fundamental module $V({\omega_k})$ such vectors are parametrized
by
\begin{equation}\label{chain}
f_{1,l_1}\dots f_{k,l_k}v_{\omega_k}, n+1\le l_k<\dots<l_1\le 2n
\end{equation}
(an important observation is that the number of factors is exactly $k$).
In order to prove \eqref{baschar} let us generalize the formula. Namely, we consider
the order on the elements $f_{ac}$, $1\le a<n<c\le 2n$ induced by \eqref{forder}:
\begin{equation}\label{pqorder}
f_{1,n+1},f_{2,n+1},\dots, f_{n-1,n+1}, f_{1,n+2}, f_{2,n+2},\dots, f_{n-1,n+2},\dots, f_{1,2n},\dots,f_{n-1,2n}.
\end{equation}

\begin{rem}
To be precise, we only need the elements $f_{pq}$ with $p+q\le 2n+1$ (see Remark \ref{notall}).
\end{rem}

Now for a pair $a,c$ we denote by $\mathbb{B}_{>n}^{\le ac}$ the set of all $(a,c)$-restrictions of the elements from
$\mathbb{B}_{>n}$. In other words, $\mathbb{B}_{>n}^{\le ac}$ consists of all collections $(\mathcal{B}'_{k,i})$ obtained from
a collection $(\mathcal{B}_{k,i})\in \mathbb{B}_{>n}$ by forgetting all elements greater than $f_{a,c}$
(in the order \eqref{pqorder}).

The generalization of the formula \eqref{baschar} is as follows.

Let $I$ be a subset of the set $\{1,\dots,2n\}$. We define $wt(I)=\sum_{l\in I} \epsilon_l$.
For a pair $(a,c)$ such that  $1\le a<n<c\le 2n$ we say that $I\subset \{1,\dots,2n\}$ is $(a,c)$-completable
if there exists a product $f_{1,l_1}\dots f_{k,l_k}$, $k=|I|$, $l_1>\dots >l_k$ such that its $(a,c)$-restriction
(in order \eqref{pqorder})
being applied to the highest weight vector $v_{\omega_k}$ is equal (up to a sign) to $v_I$.

We note that if $I$ is $(a,c)$-completable, then $I$ is dense, i.e. $I\cap\{1,\dots,n\}=\{1,\dots,r\}$ for some $r$.
Two examples are in order:
\begin{itemize}
\item any $I\subset \{n+1,\dots,2n\}$ is $(n-1,2n)$-completable;
\item if $(a,c)=(1,n+1)$, then the completable $I$'s are as follows: $I=\{1,\dots,|I|\}$ if $|I|>1$; $I=\{1\}$ or
$I=\{n+1\}$, if $|I|=1$.
\end{itemize}

\begin{rem}
A dense $I=(1,\dots,r,i_{r+1},\dots,i_k)$, $r< n$, $2n\ge i_{r+1}>\dots>i_k>n$
is $(a,c)$-completable if and only if
\begin{itemize}
\item $i_{r+1}\le c$,
\item if $i_{r+1}=c$, then $r+1\le a$.
\end{itemize}
\end{rem}

In order to state the generalization of formula \eqref{baschar} we need one more piece of notation.
Let
\begin{gather*}
I=(1,\dots,r_1,i_{r_1+1},\dots,i_{l(I)}),\ r_1<n<i_{l(I)}<\dots <i_{r_1+1}\le 2n,\\
J=(1,\dots,r_2,j_{r_2+1},\dots,j_{l(J)}),\ r_2<n<j_{l(J)}<\dots <j_{r_2+1}\le 2n
\end{gather*}
be two $(a,c)$-completable sets. We consider the following sequence $P(I,J)=(p_1,\dots,p_u)$
(to be compared with $P(\sigma,\tau)$, see Definition \ref{kdefinition}).

First, we replace the elements $1,\dots,r_1$ in $I$ and $1,\dots,r_2$ in $J$ with the number {2n+1}.

Second, if $l(I)>l(J)$, then $p_1=i_{l(I)}$. If $l(I)<l(J)$, then $p_1=j_{l(J)}$.
If $l(I)=l(J)$, then $p_1=\max(i_{l(I)},j_{l(J)})$. We put the set ($I$ or $J$)
containing $p_1$ to the left and the other set to the right. If $l(I)=l(J)$ and $i_{l(I)}=j_{l(J)}$, then
we compare $i_{l(I)-1}$ and $j_{l(J)-1}$, etc. until we find $i_m\ne j_m$. Then we put to the left the set containing
$\max(i_m,j_m)$. We consider the sets $I$ and $J$ as columns.
We note that the numbers in both columns (non-strictly) decrease from top to bottom.

Now we move upstairs in the left column writing the elements to $P(I,J)$ provided the left element
is no smaller than the right one. If at some point the sign got changed, we change the column and write
the corresponding element to $P(I,J)$. We then continue moving upstairs until the sign is preserved
adding the elements we pass to $P(I,J)$. If at some point the sign got changed, we change the
column, etc.

\begin{dfn}
We define $k^{ac}(I,J)$ as the number of times we change the columns in $P(I,J)$.
\end{dfn}

\begin{example}
Let $I=\{1,2,n+4,n+2,n+1\}$ , $J=\{1,n+5,n+4,n+3\}$, $n$ large enough. Assume that both $I$ and $J$ are $(a,c)$-completable.
Then we change $1$ and $2$ to $2n+1$ and we get
the following columns
\[
\begin{tabular}{ccc}
${\bf 2n+1}$ & $\ge$ & $2n+1$\\
${\bf 2n+1}$ &  $>$  & ${\bf n+5}$ \\
$n+4$ & $\le$ & ${\bf n+4}$\\
${\bf n+2}$ & $<$ & ${\bf n+3}$\\
${\bf n+1}$ &  &
\end{tabular}
\]
The set  $P(I,J)$ is equal to $\{n+1,n+2,n+3,n+4,n+5,2n+1,2n+1\}$ and $k^{ac}(I,J)=2$.
\end{example}

For a collection of numbers $\rho=(\rho_I)_I$ labeled by the cardinality $k$ subsets of $\{1,\dots,2n\}$ and
summing up to $m_k$ we use the notation
$$\bin{m_k}{[{\bf \rho}]_k}_q=\frac{(q)_{m_k}}{\prod_I (q)_{\rho_I}}.$$
We prove the following combinatorial identities parametrized by pairs $(a,c)$:
\begin{prop}
\begin{multline}\label{acbaschar}
\sum_{B \in \mathbb{B}^{ac}_{>n}}q^{\sum_{k,i}\sum_{f_{jl}\in\mathcal{B}_{k,i}}d(j,l,k,i)}x^{\la+wt(B)}\\=
\sum_{\substack{\rho_I\ge 0\\ I: (a,c)-completable\\ \sum_{|I|=k} \rho_I=m_k}} q^{\sum_{\{I,J\}}k^{ac}(I,J)\rho_I\rho_J}
\prod_{k=1}^{n-1} \bin{m_k}{[{\bf \rho}]_k}_q x^{\sum_{I} \rho_I wt(I)}.
\end{multline}
\end{prop}
\begin{proof}
The sum in the quadratic form in the right hand side is taken over all unordered pairs of $I,J$, i.e.
each pair appears only once (we note that $k^{ac}(I,J)=k^{ac}(J,I)$ and $k^{ac}(I,I)=0$).
We prove formula \eqref{acbaschar} by induction on $a,c$ in the order \eqref{pqorder}. We note that
the $(n-1,2n)$-formula coincides with the desired identity \eqref{baschar}.

Let us start with the base of induction $(a,c)=(1,n+1)$. We note that $\mathbb{B}^{1,n+1}_{>n}$ consists
of collections $(\mathcal{B}_{k,i})$ with the following entries: all entries $\mathcal{B}_{k,i}$
with $k>1$ are empty and the entries of the form $\mathcal{B}_{1,i}$ are either empty or equal to $f_{1,n+1}$.
Indeed, a non empty  entry $\mathcal{B}_{k,i}$ is equal to $f_{1,n+1}$. If $k$ is greater than one, there is no way
to extend it to an element of the form \eqref{chain}. Now let $1\le i_1<\dots<i_s\le m_1$ be the numbers
such that $\mathcal{B}_{1,i}=\{f_{1,n+1}\}$. Then by Definition \ref{dpqki} for $i$ such that
$\mathcal{B}_{1,i}=\{f_{1,n+1}\}$ the quantity $d(1,n+1,k,i)$ is equal to the number of $i'<i$ such that
$\mathcal{B}_{1,i}=\emptyset$. Therefore,
\begin{multline*}
\sum_{B \in \mathbb{B}^{1,n+1}_{>n}}q^{\sum_{k,i}\sum_{f_{jl}\in\mathcal{B}_{k,i}}d(j,l,k,i)}x^{\la+wt(B)}\\=
\sum_{\rho_1+\rho_{n+1}= m_1} \bin{m_1}{\rho_{1},\rho_{n+1}}_q x^{\la+\rho_{n+1}wt(f_{1,n+1})}.
\end{multline*}
Here $\rho_{n+1}$ denotes the number of nonempty $\mathcal{B}_{1,i}$ (i.e. the number of $i$ such that
$\mathcal{B}_{1,i}=\{f_{1,n+1}\}$). We note that $f_{1,n+1}$ maps $v_1$ to $v_{n+1}$, so $\rho_{n+1}$ is
equal to the number of times the basis vector $v_{n+1}$ shows up.

Before going to the general induction step, let us write down explicitly the formula for the case $(a,c)=(n-1,n+1)$.
In this case the $(a,c)$-completable $B$ are as follows: either $\mathcal{B}_{k,i}=\emptyset$ or
$\mathcal{B}_{k,i}=\{f_{k,n+1}\}$. In other words, $(a,c)$-completable sets $I$ are either $\{1,\dots,k\}$ or
$\{1,\dots,k-1,n+1\}$. It is easy to see that for all such $I,J$ we have $k^{(n-1,n+1)}(I,J)=0$. According to
Definition  \ref{dpqki} we obtain the following equality
\begin{multline*}
\sum_{B \in \mathbb{B}^{n-1,n+1}_{>n}}q^{\sum_{k,i}\sum_{f_{jl}\in\mathcal{B}_{k,i}}d(j,l,k,i)}x^{\la+wt(B)}\\=
\sum_{\substack{\rho_1+\rho_{n+1}=m_1\\ \rho_{12}+\rho_{1,n+1}=m_2\\ \dots \\ \rho_{1,\dots,n-1}+\rho_{1,\dots,n-2,n+1}=m_{n-1}}}
\bin{m_1}{\rho_{1},\rho_{n+1}}_q\dots  \bin{m_{n-1}}{\rho_{1,2,\dots, n-1},\rho_{1,\dots,n-2,n+1}}_q \\
\times x^{\la +\sum_{k=1}^{n-1} \rho_{1,\dots,k-1,n+1}(\epsilon_{n+1}-\epsilon_k)}.
\end{multline*}
In fact, according to Definition \ref{dpqki} in  order to compute the left hand side of the above formula we count the
number of pairs $i_1,i_2$ such that
$i_1<i_2$ and  $\mathcal{B}_{k,i_1}=\emptyset$, $\mathcal{B}_{k,i_2}=\{f_{k,n+1}\}$. Clearly, the generating function
is equal to the product of $q$-binomial coefficients given in the right hand side of the formula.

We now proceed by induction. Assume that equality \eqref{acbaschar} holds for a pair $(a,c)$. There are two
separate cases: $a+c<2n+1$ and $a+c=2n+1$. We work out the first case, the second is very similar.

The element $f_{a,c}$ is followed by the element $f_{a+1,c}$ (in our case $a+1+c$ is still no larger than $2n+1$).
We consider all admissible places $\mathcal{B}_{k,i}$. The corresponding vectors
$v_I$ satisfy the following properties: $I\cap\{1,\dots,n\}=\{1,\dots,a+1\}$ and
$I\cap\{n+1,\dots,2n\}\subset \{n+1,\dots,c-1\}$. Let us consider the set of $(a,c)$-completable $I$ and
the corresponding variables $\rho_I$ showing up in the formula \eqref{acbaschar}. We also consider all sets
$B=(\mathcal{B}_{k,i})$ showing up in the left hand side of \eqref{acbaschar} (for the pair $(a,c)$)
with fixed numbers $\rho_I$ (for all $(a,c)$-completable $I$). In order to pass to the $(a+1,c)$ case we
apply the operator $f_{a+1,c}$ to some of the $\mathcal{B}_{k,i}$. As a result some of the $(a,c)$-completable
and $(a+1,c)$-admissible $I$ got replaced with $I'=I\setminus \{a+1\}\cup\{c\}$. By induction we know that
the sum of the terms
\[
q^{\sum_{k,i}\sum_{f_{jl}\in\mathcal{B}_{k,i}}d(j,l,k,i)}x^{\la+wt(B)}
\]
for all $B$ with the fixed numbers $\rho_I$ for all $(a,c)$-completable $I$ satisfying
$\sum_{|I|=k} \rho_I=m_k$ is equal to
\begin{equation}\label{rhs}
q^{\sum_{\{I,J\}}k^{ac}(I,J)\rho_I\rho_J}
\prod_{k=1}^{n-1} \bin{m_k}{[{\bf \rho}]_k}_q x^{\sum_{I} \rho_I wt(I)}
\end{equation}
(the sum in the quadratic form is taken over all unordered pairs of $(a,c)$-completable $I,J$, i.e.
each pair appears only once). We want to control what is happening with the expression \eqref{rhs}
after we pass from the $(a,c)$-case to the $(a+1,c)$-case. For each $(a,c)$-completable and $(a+1,c)$-admissible
$I$ the $\rho_I$ positions of $B$ with $\mathcal{B}_{k,i}$ producing vector $v_I$ are divided into two
parts: the first part consists of the positions where $\mathcal{B}_{k,i}$ does not change (hence, $I$ does not change);
the second part consists of positions with $\mathcal{B}'_{k,i}=\mathcal{B}_{k,i}\cup \{f_{a+1,c}\}$ or, equivalently,
$I'=I\setminus\{a+1\}\cup\{c\}$.  We denote the number of the positions from the second group by
$\rho_{I'}$; thus, the number of positions of the first group is $\rho_I-\rho_{I'}$.
The change from $(a,c)$ to $(a+1,c)$ amounts in multiplication by the expression
\begin{equation}\label{factor}
\prod_I \bin{\rho_I}{\rho_{I'}}_q q^{\sum_{J<I} \rho_{I'}(\rho_J-\rho_{J'})},
\end{equation}
where the product is taken over all $(a,c)$-completable and $(a+1,c)$-admissible $I$ and the
sum in the exponent of $q$ is taken over all pairs of $(a,c)$-completable and $(a+1,c)$-admissible
$I, J$ with the condition $J<I$ coming from the rule formulated above Definition \ref{dpqki}.
In fact, $\rho_J-\rho_{J'}$ is the number of $(a+1,c)$-admissible places smaller than $I$, which are not occupied by $f_{a+1,c}$.
Multiplying  the $q$-binomial part of \eqref{factor} with the corresponding term of the $(a,c)$ formula \eqref{acbaschar} we obtain
\begin{equation}
\prod_{k=1}^{n-1}\bin{m_k}{[{\bf \rho}]_k}_q
\prod_{\substack{I\\ I: (a,c)-completable\\ I: (a+1,c)-admissible}} \bin{\rho_I}{\rho_{I'}}_q=
\prod_{k=1}^{n-1}\bin{m_k}{[{\bf \bar \rho}]_k}_q,
\end{equation}
where
\begin{itemize}
\item in the product $\prod_{k=1}^{n-1}\bin{m_k}{[{\bf \rho}]_k}_q$ the component of the vector $[{\bf \rho}]_k$ are
numbers $\rho_I$ parametrized by $(a,c)$-completable $I$;
\item in the product $\prod_{k=1}^{n-1}\bin{m_k}{[{\bf \bar\rho}]_k}_q$ the component of the vector $[{\bf \bar\rho}]_k$ are
numbers $\bar\rho_I$ parametrized by $(a+1,c)$-completable $I$.
\end{itemize}
We note that
\begin{equation}\label{J'}
\bar\rho_J=\begin{cases} \rho_J-\rho_{J'}, &  J \text{ is } (a+1,c) \text{ admissible},\\
\rho_J, &  J \text{ is not } (a+1,c)-\text{admissible}. \end{cases}
\end{equation}
Therefore,
\begin{equation}\label{barrho}
q^{\sum_{J<I} \rho_{I'}(\rho_J-\rho_{J'})}=q^{\sum_{J<I} \bar \rho_{I'}\bar \rho_J}.
\end{equation}

Let us now write down the function $k^{a+1,c}$. First, if $I$ and $J$ are not $(a+1,c)$-admissible, then
$k^{a+1,c}(I,J)=k^{a,c}(I,J)$.  Second, if $I$ is $(a+1,c)$-admissible, but $J$ is not $(a+1,c)$-admissible, then
\[
k^{a+1,c}(I,J)=k^{a+1,c}(I',J)=k^{a,c}(I,J).
\]
Third, if both $I$ and $J$ are $(a+1,c)$-admissible, then
\begin{gather*}
k^{a+1,c}(I,J)=k^{a,c}(I,J),\\
k^{a+1,c}(I',J')=k^{a,c}(I,J),\\
k^{a+1,c}(I',J)=
\begin{cases}
k^{a,c}(I,J),& I<J\\ k^{a,c}(I,J)+1,& I>J.
\end{cases}
\end{gather*}
Using formula \eqref{barrho}, we derive the $(a+1,c)$-case of formula \eqref{acbaschar}.
\end{proof}

\begin{cor}
Equality \eqref{baschar} holds true.
\end{cor}
\begin{proof}
Formula \eqref{baschar} is obtained from formula \eqref{acbaschar} for $(a,c)=(n-1,2n)$ via substitution
$\rho_I=r_{w_0\sigma}$, where $w_0$ is the longest element in the symmetric group $S_{2n}$ (note that all
$(n-1.2n)$-completable $I$ satisfy $I\subset \{n+1,\dots,2n\}$).
\end{proof}

\begin{thm}\label{main}
We have the following formula for the character of the Weyl module $W(\lambda)$ over $\msl_n$:
\begin{equation}
{\rm ch}W(\lambda)=\sum_{\sum_{|\sigma|=k} r_\sigma=m_k}
q^{\sum_{\sigma<\tau}k(\sigma,\tau)r_{\sigma}r_{\tau}}\prod_{i=1}^n x_i^{\sum_{\sigma\ni i} r_\sigma}
\prod_{k=1}^{n-1} \bin{m_k}{[\br]_k}_q.
\end{equation}
\end{thm}
\begin{proof}
The character of $\mathcal{F}$-basis is equal to the character of the left hand side of \eqref{ineq} multiplied by
$(q)_{\lambda}$. Therefore Lemma \ref{ineqchar} and Lemma \ref{inverceineqchar} imply our Theorem.
\end{proof}

\begin{cor}\label{cormain}
$\mathbb{W} \simeq \overline{\mathbb{W}}\simeq \mathcal{M}$, i.e. the algebra $\mathbb{W}$ (a.k.a. $\mathcal{M}$) is generated
by the dual fundamental Weyl modules with the set of quadratic relations
\eqref{snakeplueckerequation}.
\end{cor}

\appendix

\section{Symplectic rank two case}\label{app}
We work out only the case of type $A$ because there is a nice combinatorics of minors describing Pl\"ucker relations.
So we can write down all relations in explicit way.
However the methods from the main body of this paper can be applied to Weyl modules over Lie algebras of other types.
To illustrate this let us compute the characters of Weyl modules in type $C_2$. Both fundamental modules satisfy conditions of Lemma
\ref{fundamentalstructure}.
The fundamental module $V(\omega_1)^*$ has a basis $\{X_{\pm\epsilon_1},X_{\pm\epsilon_2}\}$,
the fundamental module $V(\omega_2)^*$ has a basis $\{X_{\pm\epsilon_1\pm\epsilon_2},X_0\}$.
Hence the Weyl module $W(\omega_1)^*$ has a basis $\{X_{\pm\epsilon_1}^{(l)},X_{\pm\epsilon_2}^{(l)}\}$,
the Weyl module $W(\omega_2)^*$ has a basis $\{X_{\pm\epsilon_1\pm\epsilon_2}^{(l)},X_0^{(l)}\}$,
$l \in \mathbb{N}\cup\{0\}$. Let
\[X_{\mu}(s)=\sum_{l\ge 0} X_{\mu}^{(l)}s^l.\]
The algebra $\mathbb{V}=\bigoplus_{\lambda \in P_+}V(\lambda)^*$ satisfies the following relations:
\begin{equation}\label{relarionsC222}
X_{\epsilon_1+\epsilon_2}X_{-\epsilon_1-\epsilon_2}+ X_{\epsilon_1-\epsilon_2}X_{-\epsilon_1+\epsilon_2}+X_0^2=0;
\end{equation}
\begin{equation}\label{relarionsC212}
X_{\epsilon_1}X_{-\epsilon_{1}+\epsilon_2}+ X_{-\epsilon_1}X_{\epsilon_{1}+\epsilon_2}+X_{\epsilon_2}X_0=0
\end{equation}
and three relations obtained from \eqref{relarionsC212} by the Weyl group.
Therefore we have the following relations in algebra $\mathbb{W}$:

\begin{equation}\label{relarionsC222semiinf}
X_{\epsilon_1+\epsilon_2}(s)X_{-\epsilon_1-\epsilon_2}(s)+ X_{\epsilon_1-\epsilon_2}(s)X_{-\epsilon_1+\epsilon_2}(s)+X_0(s)^2=0;
\end{equation}

\begin{equation}\label{relarionsC212semiinf}
X_{\epsilon_1}(s)X_{-\epsilon_{1}+\epsilon_2}(s)+ X_{-\epsilon_1}(s)X_{\epsilon_{1}+\epsilon_2}(s)+X_{\epsilon_2}(s)X_0(s)=0
\end{equation}
and three relations obtained from \eqref{relarionsC212semiinf} by the Weyl group.
Let $\mathcal R$ be the set of nine-tuples of numbers $(r_{\pm\epsilon_1},r_{\pm\epsilon_2},r_{\pm\epsilon_1\pm \epsilon_2}, r_0)$
such that $r_{\epsilon_1}+r_{-\epsilon_1}+r_{\epsilon_2}+r_{-\epsilon_2}=m_1$,
$\sum r_{\pm \epsilon_1\pm \epsilon_2}+r_{0}=m_2$.
For $\bar r \in \mathcal{R}$ define:
\[b(\bar r)=r_{\epsilon_2-\epsilon_1}r_{-\epsilon_2+\epsilon_1}+(r_{\epsilon_2+\epsilon_1}+r_{-\epsilon_2+\epsilon_1})r_{-\epsilon_1}+
(r_{\epsilon_2+\epsilon_1}+r_{\epsilon_2-\epsilon_1})r_{-\epsilon_2},\]
\[\varsigma_1(\bar r)=
r_{\epsilon_1}-r_{-\epsilon_1}+r_{\epsilon_1+\epsilon_2}+r_{\epsilon_1-\epsilon_2}-r_{-\epsilon_1+\epsilon_2}-r_{-\epsilon_1-\epsilon_2},\]
\[\varsigma_2(\bar r)=
r_{\epsilon_2}-r_{-\epsilon_2}+r_{\epsilon_1+\epsilon_2}-r_{\epsilon_1-\epsilon_2}+r_{-\epsilon_1+\epsilon_2}-r_{-\epsilon_1-\epsilon_2}\]

Therefore for $\lambda=m_1\omega_1+m_2\omega_2$ we have the following inequality of characters:
\begin{multline}\label{charaterC2}
(q)_{\lambda}^{-1}\sum_{\bar r \in \mathcal{R}}
\prod q^{b(\bar r)} x_1^{\varsigma_1(\bar r)} x_2^{\varsigma_2(\bar r)}\\
\binom{m_1}{r_{\epsilon_1},r_{-\epsilon_1},r_{\epsilon_2},r_{-\epsilon_2}}_q\binom{m_2}{r_{\pm \epsilon_1\pm \epsilon_2},r_{0}}_q
\geq {\rm ch}\mathbb{W}(\lambda).
\end{multline}

For a tuple of pairwise distinct elements
$(\zeta_{k,i})$, $1 \leq i \leq m_k$, $k=1,2$ we consider the tensor product:
\[V(\lambda)_{(\zeta_{k,i})}=\bigotimes_{k=1}^{2}\bigotimes_{i=1}^{m_k}V(\omega_k)_{\zeta_{k,i}}.\]
Let $v$ be the highest weight vector of this tensor product.

The negative root vectors of $\mathfrak{sp}_4$ are of the form $f_{-2\epsilon_1},f_{-\epsilon_2-\epsilon_1},f_{-2\epsilon_2},f_{-\epsilon_2+\epsilon_1}$.
We consider the set $\mathbb{B}$ of collections $B=(\mathcal{B}_{k,i})$, $k=1,2$, $1 \leq i \leq m_k$, where
$\mathcal{B}_{1,i}$ is either empty or contains one of the following elements: $f_{-\epsilon_2-\epsilon_1}$,
$f_{-2\epsilon_2}$, $f_{-\epsilon_2+\epsilon_1}$; $\mathcal{B}_{2,i}$ is equal to one of five following sets:
\[\emptyset, \{f_{-2\epsilon_1}\},\{f_{-\epsilon_2-\epsilon_1}\},\{f_{-2\epsilon_2}\},\{f_{-2\epsilon_1},f_{-2\epsilon_2}\}.\]
Define the following order on elements $f_{\alpha}$:
\[f_{-2\epsilon_1}<f_{-\epsilon_2-\epsilon_1}<f_{-2\epsilon_2}<f_{-\epsilon_2+\epsilon_1}.\]
For any $B \in \mathbb{B}$ the strict $\alpha$-restriction $B'$ of
$B$ is defined by $B'=(\mathcal{B}_{k,i}')$, where  $\mathcal{B}_{k,i}'=\mathcal{B}_{k,i}\cap \{f_{\gamma}, f_{\gamma}<f_{\alpha}\}$.
We define $\alpha$-admissible pairs in the following way. A pair $(k,i)$ is $\alpha$-admissible if the $\alpha$-restriction
$\mathcal{B}_{k,i}'$ is empty or $\alpha=-2\epsilon_2$, $k=2$ and $\mathcal{B}_{k,i}'=\{f_{-2\epsilon_1}\}$.
We define an order on $\alpha$-admissible pairs in the following way: $(k,i)<(k,j)$ if $i<j$ and $(2,i)<(1,j)$.
Then for $f_{\alpha} \in \mathcal{B}_{k,i}$ define $d(\alpha,k,i)$ as the number of $\alpha$-admissible pairs $(k',i')<(k,i)$
such that $f_{\alpha}\notin \mathcal{B}_{k',i'}$.
Then:
\[\mathcal{F}(B)=\prod_{f_{\alpha}\in \mathcal{B}_{k,i}}f_{\alpha}\otimes t^{d(\alpha,k,i)}v.\]
\begin{prop}
The set $\{\mathcal{F}(B)\}$ is a basis of $V(\lambda)_{(\zeta_{k,i})}$.
\end{prop}
\begin{proof}
The proof is analogous to the proof of Proposition \ref{basisevaluation}.
\end{proof}
Moreover by direct computations we obtain that the character of the set $\mathcal{F}(\mathbb{B})$
is equal to the left hand side of \eqref{charaterC2} multiplied by
$(q)_{\lambda}$. Thus we get the character formula for Weyl modules over  Lie algebras of type $C_2$.
In order to generalize the combinatorial construction to the higher rank symplectic algebras one
will need the monomial bases constructed in \cite{FFL2}.

\section*{Acknowledgments}
We are grateful to Alexander Braverman and Michael Finkelberg for useful discussions and explanations.
We thank Igor Makhlin for bringing our attention to \cite{MS}.
The work was partially supported by the grant RSF-DFG 16-41-01013.

\end{document}